\documentclass[10pt]{article}
\usepackage{fullpage,epsfig,graphics,amsbsy,amssymb,cancel,bbm}
\usepackage{psfrag,hyperref,color,mathrsfs}
\usepackage{graphicx}
\usepackage{amsfonts}
\usepackage{amssymb,amsmath,amscd}
\usepackage{tikz}
\usetikzlibrary{arrows,chains,matrix,positioning,scopes}

\makeatletter
\tikzset{join/.code=\tikzset{after node path={%
\ifx\tikzchainprevious\pgfutil@empty\else(\tikzchainprevious)%
edge[every join]#1(\tikzchaincurrent)\fi}}}
\makeatother

\tikzset{>=stealth',every on chain/.append style={join},
         every join/.style={->}}
\tikzset{
    >=stealth',
    punkt/.style={
           rectangle,
           rounded corners,
           draw=black, very thick,
           text width=6.5em,
           minimum height=2em,
           text centered},
    pil/.style={
           ->,
           thick,
           shorten <=2pt,
           shorten >=2pt,}
}

\newcommand{\BS}{\boldsymbol}
\newcommand{\BB}{\mathbb}
\newcommand{\SF}{\mathsf}
\newcommand{\FR}{\mathfrak}
\newcommand{\SL}{\textsl}

\def\M{{\cal M}}

\def\E{{\cal E}}
\def\Mod{\textsf{M}}

\def\A{\textsf{A}}

\newcommand{\bea}{\begin{eqnarray}}
\newcommand{\eea}{\end{eqnarray}}
\newcommand{\nn}{\nonumber}
\newcommand{\Tr}{\textrm{Tr}}

\newcommand{\sbullet}{\textrm{\tiny{\textbullet}}}
\newcommand{\udl}{\underline}

\def\ga{\alpha}
\def\gb{\beta}
\def\gc{\gamma}
\def\gd{\delta}

\DeclareMathAlphabet{\mathpzc}{OT1}{pzc}{m}{it}
\def\vara{{\large\textrm{$\mathpzc{a}$}}}

\newtheorem{theorem}{Theorem}[section]

\newtheorem{proposition}[theorem]{Proposition}

\newenvironment{proof}[1][Proof]{\begin{trivlist}
\item[\hskip \labelsep {\bfseries #1}]}{\end{trivlist}}
\newenvironment{definition}[1][Definition]{\begin{trivlist}
\item[\hskip \labelsep {\bfseries #1}]}{\end{trivlist}}
\newenvironment{example}[1][Example]{\begin{trivlist}
\item[\hskip \labelsep {\bfseries #1}]}{\end{trivlist}}
\newenvironment{remark}[1][Remark]{\begin{trivlist}
\item[\hskip \labelsep {\bfseries #1}]}{\end{trivlist}}

\newcommand{\qed}{\nobreak \ifvmode \relax \else
      \ifdim\lastskip<1.5em \hskip-\lastskip
      \hskip1.5em plus0em minus0.5em \fi \nobreak
      \vrule height0.5em width0.5em depth0.00em\fi}

\def\nb{\mathfrak{n}}

\begin{document}
\thispagestyle{empty}
\begin{flushright} \small
UUITP-27/11\\
 \end{flushright}
\smallskip
\begin{center} \Large
{\bf Knot Weight Systems from Graded Symplectic Geometry}
  \\[12mm] \normalsize
{\bf Jian Qiu$^a$ and Maxim Zabzine$^b$} \\[8mm]
 {\small\it
 ${}^a$I.N.F.N. and Dipartimento di Fisica\\
     Via G. Sansone 1, 50019 Sesto Fiorentino - Firenze, Italy\\
   \vspace{.3cm}
     ${}^b$Department of Physics and Astronomy,
     Uppsala university,\\
     Box 516,
     SE-751\;20 Uppsala,
     Sweden\\
 }
\end{center}
\vspace{7mm}

\begin{abstract}
\noindent We show that from an even degree symplectic $NQ$-manifold, whose homological vector field $Q$ preserves the symplectic form, one can construct a weight system for tri-valent graphs with values in the $Q$-cohomology ring, satisfying the IHX relation. Likewise, given a representation of the homological vector field, one can construct a weight system for the chord diagrams, satisfying the IHX and STU relations.
 Moreover we show that the use of the 'Gronthendieck connection' in the construction is essential in making the weight system dependent only on the choice of the $NQ$-manifold and its representation.
\end{abstract}

\eject

\section{Introduction}
The Vassiliev knot invariants \cite{Vassiliev,Bar-Natan_I} compute the cohomology of the space of embeddings of one or several circles into $S^3$, and can be conveniently kept track of by means of the so called \emph{chord diagrams}. The precise definition of the chord diagrams will appear in sec.\ref{sec_graph}, and the \emph{weight system} for knot invariants is basically the scheme of assigning weights to the chord diagrams. The key criterium that a weight system has to meet is the IHX relation and STU relation; the IHX relation can be thought of as some kind of Jacobi identity and the STU relation is analogues to the matrix representation of these Jacobi identities. Thus Lie algebras and their matrix representations furnish the most important class of weight systems. While some preliminary calculation by Bar-Natan (see ref.\cite{Not_happy}) at low dimension showed that the Lie algebra weight systems exhaust the Vassiliev knot invariants, it has been established later \cite{Vogel} that at sufficient high dimension, the Lie algebras fail to provide all the weight systems.

In the work of Rozansky and Witten \cite{RW96}, and later in greater detail by Sawon \cite{Sawon}, it was pointed out that a hyperK\"ahler manifold plus a holomorphic vector bundle on it also gives a valid weight system. Kapranov \cite{Kapranov} showed that one may associate an $L_{\infty}$-algebra structure to a hyperK\"ahler manifold and one is led to the realization that an $L_{\infty}$-algebra structure in conjunction with a proper notion of representation provides a wider class of weight systems, which may take value in certain rings. As an example, the Rozansky-Witten (RW) weight-system takes value in the Dolbeault cohomology ring. Kontsevich also pointed out in ref.\cite{Kontsevich_RW} that this construction can be understood as applying the Chern-Weil construction to a $Q$-structure to form secondary characteristic classes.

In an earlier publication \cite{WilsonLoop} of the authors, we have built up a class of topological field theories (TFT's) that are associated with the $L_{\infty}$-algebra structures in the same way the Chern-Simons theory is associated with a Lie algebra. One can use some path integral manipulations and Ward identities to prove that these $L_{\infty}$ structures do give valid weight systems at the level of 'physics rigour'. Thus we feel that we owe the reader a purely mathematical and more careful account of our formalism. Besides, the treatment of an $L_{\infty}$-structure arising from a curved $NQ$-manifold is more subtle than we were able to appreciate at the writing of ref.\cite{WilsonLoop}, and we shall rectify this remiss in this paper.

 To summarize, the main results of the paper are theorems \ref{main_A} and \ref{main_B}. We take an even degree symplectic $N$-manifold $(\M,\omega)$ and use an exponential map to identify a neighbourhood round any point with a graded vector space $\mathbbm{R}$ of the same dimension. Theorem \ref{main_B} is a recipe for constructing cocycles of Lie algebra of vector fields on $\M$ from a cocycle of vector fields on $\mathbbm{R}$. And the use of the \emph{Grothendieck connection} in the formalism is to make the cohomology class of the cocycle independent of the connection chosen to construct the exponential map.
When $\M$ has an $\omega$-preserving homological vector field $Q$, and a graded bundle $\E$ over $\M$ equipped with a lift $Q^{\uparrow}$ of $Q$, one can construct two objects $\widehat Q$ and $\widehat T$ that give rise to an $L_{\infty}$-structure (up to $Q$-exact terms) and its representation. The object $\widehat Q$ is a vector field on $\mathbbm{R}$ and can be lifted to a Hamiltonian function $\Theta$, since $\mathbbm{R}$ inherits a symplectic structure from $\M$. Theorem \ref{main_A} states that, one can use the third (resp. first) Taylor coefficient of $\Theta$ (resp. $\widehat T$) as vertices and form a weight-system for the chord diagrams with values in the $Q$-cohomology ring of $\M$. The result is again independent of any non-canonical choices: connections, trivializations and the like.

As examples, we generalize the RW weight system based on hyperK\"ahler manifolds to include also the holomorphic symplectic manifolds. We also construct an equivariant version of the previous two cases, which illustrates the subtle role played by the Grothendieck connection. This is also an example that is not treated (as far as we know) by others. The necessity of the Grothendieck connection was stumbled upon when we tried to show in ref.\cite{RWCS} the independence of the cohomology class on the metric. We will formulate the result of ref.\cite{RWCS}, which is a piece of physically minded work, in a more mathematical manner here. Let us explain here the rough idea of obtaining an $L_{\infty}$-structure from an $NQ$ manifold. It is well known by now that if one expands the vector field $Q$ into a formal power series around one of its zeros, then the $n^{th}$ order Taylor coefficients give the structure constant of an $n$-ary bracket, and all these brackets combine to yield an $L_{\infty}$ structure.
If one wishes to expand round a generic point of $\M$, then one needs to manually subtract the $0^{th}$ coefficient and consider $L_{\infty}$ structure up to $Q$-exact terms. A large portion of this paper is about using the Grothendieck connection to subtract the $0^{th}$ term in a covariant way. This is the part that does not feature in Kapranov's construction, because there $Q$ is the Dolbeault operator $Q=\bar\partial$ and after applying the holomorphic exponential map for the Taylor expansion of $Q$, the $0^{th}$ coefficient is zero automatically, for details, see sec.\ref{sec_WfGViQc}.

The paper is organized as follows, we first present in sec.\ref{sec_TGC} the formula for the Gronthendieck connection and explore its geometrical meaning, in particular the Bott-Haefliger construction. Then in sec.\ref{sec_CC}, we review the Chevalley-Eilenberg (CE) complex of Lie algebra of vector fields and give the recipe for converting a CE cochain on a flat space to a cochain on a curved space with value in the ring of functions. In sec.\ref{sec_EM}, we give the formula for the exponential map and Gronthendieck connection in this special case and verify all properties claimed in sec.\ref{sec_TGC} and \ref{sec_EM}. Section \ref{sec_graph} contains the recipe of the mapping between CE complex and graph complex, which is a useful book keeping tool for the computation. Also our main theorems are stated in that section. Finally, in sec.\ref{sec_EX}, we give plenty of examples to endow substance to our construction.\bigskip

\noindent{\it Acknowledgement}:
The authors would like to thank E. Getzler for discussions on the problem of trivialization dependence of weight-systems, from which the current paper is conceived.
 M.Z. thanks the Simons center for geometry and physics where part of this work was carried out.
The research of M.Z. is supported by VR-grant 621-2008-4273.

\tableofcontents

\section{The Grothendieck Connection}\label{sec_TGC}
{\color{black}As a general preamble, a graded manifold $\M$ consists of a smooth manifold $M$ with a sheaf of freely generated graded commutative algebra on $M$. The \emph{section}s of this sheaf, denoted $C^{\infty}(\M)$, are often called the \emph{function}s on the graded manifold. The underlying manifold $M$ is called the reduced manifold or body of $\M$, denoted as $M=|\M|$. We will employ the 'functor of points' perspective in handling the graded manifolds, which allows one to treat the generators of the local graded algebra as coordinates and talk about 'points' on $\M$, for more expositions regarding this, see ref.\cite{NoteonSusy}.}

Consider a non-negatively graded manifold ($N$-manifold) $\M$, which is locally of the type
\bea \mathbbm{R}=\BB{R}^n\times \BB{R}^{n_1}[1]\times \BB{R}^{n_2}[2]\times\cdots\nn\eea
Choose the coordinates of $\M$ to be $x^A$, and the Greek letters $\xi^A$ for the corresponding coordinates of $\mathbbm{R}$. We will reserve $A,B,\cdots$ for the general coordinates, while $\mu,\nu,\cdots$ for the zero degree coordinates, namely $x^{\mu}$ is the coordinate of the body $|\M|$ and $\xi^{\mu}$ is the coordinate of $\BB{R}^n$. We will also borrow a general relativity jargon by calling $x^A$ \emph{curved} coordinates and $\xi^A$ \emph{flat} coordinates.

In a neighborhood $U\subset|\M|$, there is an isomorphism of algebras
\bea \varphi:\,C^{\infty}(\M\big|_{U})\to C^{\infty}(\mathbbm{R}).\label{local_iso}\eea
If we pick locally a family of such isomorphisms $\phi_x$, depending smoothly on $x$, such that the origin of $\mathbbm{R}$ is mapped to $x$ by $\phi_x$ (the 'functor of points' view is being enforced here!), then we can
consider the following differential operator\footnote{Unless otherwise declared, all derivatives are left derivatives; the right derivatives differ from the left ones by a sign $x^A\overleftarrow{\partial}_A=(-1)^{|x^A|}$.}, acting on the functions of $\M\big|_{U}\times\mathbbm{R}$
\bea D=dx^A\frac{\partial}{\partial x^A}-dx^A\Big[\frac{\partial\phi_x}{\partial x}\Big]^{\;B}_A\Big[\big(\frac{\partial\phi_x}{\partial\xi}\big)^{-1}\Big]_B^{\;C}\frac{\partial}{\partial\xi^C}
= dx^A\big(\frac{\partial}{\partial x^A}-\FR{G}_A^{\;C}\frac{\partial}{\partial\xi^C}\big).\label{Gro_conn}\eea
{\color{black}The second term in $D$ is called the \emph{Grothendieck connection}, it is a 1-form defined on the total space of a bundle $\FR{B}$ to be introduced shortly, along with its other properties. For now we just take its expression at its face value and we have}
\begin{proposition}\label{prop_flat}The Grothendieck connection is flat
\bea [u\cdot D,v\cdot D]=u\cdot D\,v\cdot D-(-1)^{|u||v|}v\cdot D\,u\cdot D=[u,v]\cdot D,~~~u,v\in\textrm{vect}(\M),~~u\cdot D=u^AD_A\label{flatness}.\eea
\end{proposition}
\begin{proof} 1: A direct calculation does the job\qed\end{proof}

\noindent {\bf Proof} 2: We present a second proof which exhibits the Grothendieck connection as the pull back of a left invariant Maurer-Cartan form $g^{-1}dg$ on a pseudogroup, which would immediately imply the flatness. The following is a special case of a construction due to Bott and Haefliger \cite{BottHaefliger}, who used it to construct characteristic classes for foliations. The ensuing discussion is valid for graded manifolds, but for first reading, one may consider only a smooth manifold.

\subsection{Bott-Haefliger Construction}\label{sec_BHC}
For an $N$-manifold $\M$, let $M$ be its body. For each $x\in\M$, we can attach a bundle (the \emph{Haefliger}) structure. Over a point $x$, the fibre consists of local isomorphisms $\varphi^{-1}_{x,U}$, as in Eq.\ref{local_iso}, defined on some open neighborhood $U\in M$ containing the body of $x$: $U\ni x|_{M}$, such that the point $x$ is mapped to the origin of $\mathbbm{R}$. The pseudogroup $\Gamma$ consists of diffeomorphisms between open sets in $\mathbbm{R}$, that is, between $C^{\infty}(\mathbbm{R}\big|_U)$ and $C^{\infty}(\mathbbm{R}\big|_V)$, $U,V\subset\BB{R}^n$. And let $\Gamma_0$ consist of those diffeomorphisms fixing $0\in\mathbbm{R}$. Note $\Gamma_0$ acts transitively on the fibre, so we have a principal bundle structure\footnote{To avoid difficulty with infinite dimension, one can consider first the bundle $\FR{B}^k$ of $k$-jets of local diffeomorphisms with structure group the $k$-jets $\Gamma^k_0$, then take the inverse limit.}
\bea &&\FR{B}\longleftarrow \Gamma_0\nn\\
&&\hspace{.1cm}\big{\downarrow}\nn\\
&&\M\nn\eea

For two points $x$ and $y$ nearby in $M$ and two such isomorphisms $\varphi^{-1}_{x,U},\varphi^{-1}_{y,V}$ (each of which is a point of the fibre over $x$ and $y$),
there is an element $g\in\Gamma$ defined on $\varphi^{-1}_{x,U}(U)\cap \varphi^{-1}_{y,V}(V)$ relating $\varphi^{-1}_{x,U}$ and $\varphi^{-1}_{y,V}$:
\bea \varphi^{-1}_{y,V}=g^{-1}\circ \varphi^{-1}_{x,U}\label{trans},\eea
where $\circ$ denotes composition and we have used $g^{-1}$ for later convenience. Eq.\ref{trans} allows us to identify $\FR{B}\big|_W$ with the pseudogroup $\Gamma$ for small enough $W$. To do so, pick a point $x$ and a fiducial isomorphism $\varphi_{x,U}$, let $W\subset U$ be a small open neighbourhood of $x|_{M}$, the identification $\psi:~\FR{B}\big|_W\rightarrow \Gamma$ goes as
\bea \forall y,~y|_M \in W\subset V~\textrm{and}~\varphi_{y,V},~~\textrm{let}~~\psi(\varphi_{y,V})=g,~~\textrm{where $g$ satisfies}~\varphi^{-1}_{y,V}=g^{-1}\circ \varphi^{-1}_{x,U}.\nn\eea
If we had started from a different fiducial $\varphi^{-1}_{x',U'}$, with $\varphi^{-1}_{x,U}=h^{-1}\circ \varphi^{-1}_{x',U'}$, then clearly
\bea \psi'=h\circ \psi\label{left_trans}.\eea
We can consider the pull back through $\psi$ of a left invariant 1-form $\FR{G}=\psi^*(g^{-1}dg)$, $g\in \Gamma$. One can see that this is a well defined 1-form on $\FR{B}$, i.e. independent of the choice $\varphi_{x,U}$.  Indeed, the change of $\varphi_{x,U}$ multiplies (rather, composes) an element $h\in\Gamma$ to $\psi$ from the left as in Eq.\ref{left_trans}. Since $d$ here does not act on $h$, $\psi^*(g^{-1}dg)$ is unchanged and $\FR{G}$ is a well-defined 1-form of $\FR{B}$. We shall write $\psi^*(g)$ simply as $g$.

We show next that this 1-form leads to the Grothendieck connection Eq.\ref{Gro_conn}, thereby establishing the flatness. To pull back the 1-form $\FR{G}$ to $\M$, one chooses a local section $\phi^{-1}_x:~\M\to\FR{B}$. We shall dispense with the subscript $U$, and also remark that
$\varphi^{-1}$ is used to denote points of $\FR{B}$, but the symbol $\phi^{-1}$ is reserved for sections of $\FR{B}$.
We also write the transition function $g^{-1}$ in Eq.\ref{trans} as
\bea \phi^{-1}_y=\Gamma_{yx}\circ \phi^{-1}_x.\nn\eea
Thus we need to compute the quantity
\bea\FR{G}= dy^A\Gamma_{yx}\big(\frac{\partial}{\partial y^A}\Gamma^{-1}_{yx}\big),\nn\eea
which by construction is independent of $x$. Taking advantage of this we have
\bea \Gamma_{yx}(\partial_{y^A}\Gamma^{-1}_{yx})=\Gamma_{yx}(\partial_{y^A}\Gamma^{-1}_{yx})\big|_{x=y}=-\partial_{y^A}\Gamma_{yx}\big|_{x=y}.\nn\eea
To compute this efficiently, we do the following
\bea 0=\partial_{y^A}\phi_x=\partial_{y^A}(\phi_y\circ\Gamma_{yx})
&\Rightarrow&\bigg[(\partial_{y^A}\phi^B_y)(\eta)+\frac{\partial \Gamma^C_{yx}}{\partial y^A}\frac{\partial\phi^B_y(\eta)}{\partial\eta^C}\bigg]\bigg|_{\eta=\Gamma_{yx}(\xi)}=0.\nn\eea
Evaluating the last equation at $x=y$,
\bea (\partial_{y^A}\phi^B_y)(\xi)+\frac{\partial \Gamma^C_{yx}}{\partial y^A}\frac{\partial\phi^B_y(\xi)}{\partial\xi^C}\bigg|_{x=y}=0,\nn\eea
leads to
\bea \FR{G}^C=dy^A\Big[\frac{\partial \phi_y}{\partial y}\Big]_{A}^{\;D}\Big[\big(\frac{\partial\phi_y}{\partial \xi}\big)^{-1}\Big]^{\;C}_D.\nn\eea
This is the formula given in Eq.\ref{Gro_conn} up to a minus sign.

The sign difference comes about in the following way: to get at Eq.\ref{Gro_conn}, one represents Lie$\,{}_{\Gamma}$ on $C^{\infty}(\mathbbm{R})$, and
the action is defined
\bea (h f)(\xi)=h^{-*}f(\xi)=f(h^{-1}(\xi)),~~h\in \Gamma,~f\in C^{\infty}(\mathbbm{R}).\nn\eea
in order for it to compose correctly. This is the origin of the minus sign.

The following remark is a simple consequence of the fact that the Grothendieck connection is built from the left invariant Maurer-Cartan form.
\begin{remark}\label{change_gron_conn}
While a change of trivialization for $\FR{B}$ effects a left multiplication $g\to f\circ g$, one can move about the fibre over $y$ by a right multiplication $g\to g\circ h^{-1}_y$, $h_y^{-1}\in \Gamma_0$. Indeed, under a change of local isomorphism, $\tilde \varphi^{-1}_{y,\tilde V}=h_y\circ \varphi^{-1}_{y,V}$, the group element $g\in\Gamma$ relating $\varphi^{-1}_{y,V}$ and the fiducial $\varphi^{-1}_{x,U}$ transforms as
\bea \tilde \varphi^{-1}_{y,\tilde V}=h_y\circ \varphi^{-1}_{y,V}=h_y\circ g^{-1}\circ \varphi^{-1}_{x,U}&\Rightarrow& g\to g\circ h_y^{-1}.\nn\eea
Similarly if one computes the Grothendieck connection with a different local section $\tilde\phi_x^{-1}=h_x\circ \tilde\phi_x^{-1}$, then
\bea \FR{G}\to h_x\circ dh_x^{-1}+h_x\circ \FR{G} \circ h_x^{-1},\label{gauge_trans_gen}\eea
which gives how $\FR{G}$ depends on the fibre of $\FR{B}$.
\end{remark}
\begin{remark}
The bundle $\FR{B}$ will not in general have a global section, since, as a principle bundle, the possession of a global section would trivialize the bundle. To forestall a possible confusion, the 1-form $\FR{G}$, albeit called a connection, is \emph{not} a connection of $\FR{B}$, because it is valued in Lie$\,{}_{\Gamma}$, instead of Lie$\,{}_{\Gamma_0}$.

However, one can take a group $K\subset \Gamma_0$ maximally compact, and consider the quotient $\FR{B}/K$ with fibre $\Gamma_0/K$, which is contractible. Thus $\FR{B}/K$ possesses a global section (this is a classic result, see theorem.12.2 \cite{Steenrod}). In sec.\ref{sec_EM}, we will construct a section of $\FR{B}/SO$ using an exponential map of a flow equation, and in special cases we also have sections of $\FR{B}/GL$, and $\FR{B}/SP$, even though the latter two groups are not maximally compact.
\end{remark}

Finally, for application of weight systems for knots we need the notion of graded vector bundles over graded manifolds, defined
as sheaves of freely generated $C^{\infty}(\M)$-modules over the body of $\M$, the coordinates of the fibre are the generators of this module, see ref.\cite{NoteonSusy}.
\bea
\begin{tikzpicture}[scale=.8]\label{graded_bundle}
  \matrix (m) [matrix of math nodes, row sep=2em, column sep=2em]
    {\;Q^{\uparrow} & {\cal E} & \BB{V}\\
      Q & {\cal M} &  \\ };
 \path[->]
 (m-1-3) edge  (m-1-2)
 (m-1-2) edge  node[right] {$\small{\pi}$} (m-2-2)
 (m-1-1) edge  node[right] {$\small{\pi}_*$} (m-2-1);
\end{tikzpicture}\label{def_rep}\eea
In particular, we are interested in the case when $\M$ is an $NQ$-manifold with homological vector field $Q$ with a lifting to $Q^{\uparrow}$ acting on sections of $\E$.
The lift $Q^{\uparrow}$ can be thought of as the representation of $Q$ \cite{Vaintrob}.

We can choose a trivialization locally on a patch $\E|_U=\M|_U\times \BB{V}$, and denote the coordinate of the fibre $\BB{V}$ as $z_{\alpha}$. Locally, we can write $Q^{\uparrow}$ as
\bea Q^{\uparrow}=Q^A(x)\frac{\partial}{\partial x^A}+(-1)^{\alpha\beta+\beta}T^{\beta}_{~\alpha}(x)z_{\beta}\frac{\partial}{\partial z_{\alpha}},\label{representation}\eea
where notations like $(-1)^{\alpha}$ denote $(-1)^{|z_{\ga}|}$. Note that since $Q^{\uparrow}$ acts on the sections of $\E$, its dependence on the fibre coordinate is linear.
The nilpotency of $Q^{\uparrow}$ entails
\bea (-1)^{\beta}QT^{\beta}_{~\alpha}+(-1)^{\gamma}T^{\beta}_{~\gamma}T^{\gamma}_{~\alpha}=0\label{normalized_MC}.\eea

The splitting of $Q^{\uparrow}$ is not independent of local trivializations of $\E$: under $z_{\ga}\to z_{\ga}+z_{\gb}\epsilon^{\gb}_{~\ga}$, the $T$ matrix transforms as
\bea \delta_{\epsilon}T^{\beta}_{~\alpha}=Q\epsilon^{\beta}_{~\alpha}
+T^{\beta}_{~\gamma}\epsilon^{\gamma}_{~\alpha}-(-1)^{\beta+\gamma}\epsilon^{\beta}_{~\gamma}T^{\gamma}_{~\alpha},\label{gauge_transform}\eea
thus $T$ is \emph{not} a section of the bundle $\textrm{End}\,\E$.
The trivialization independence problem will be handled properly in sec.\ref{sec_LAoVF}.

The local model for $\E$ is $\mathbbm{R}\times \BB{V}$ for some fixed graded vector space $\BB{V}$. To build a local isomorphism of $\E|_{U}$ with $\mathbbm{R}\times \BB{V}$ we may copy Eq.\ref{local_iso} almost verbatim, except that we demand $\varphi$ be independent of the fibre of $\E$. And we use $z_{\ga}$ as the fibre coordinate of $\E$ under a trivialization, and $\zeta_{\ga}$ the corresponding coordinate in local model.

\section{Covariant Cocycles}\label{sec_CC}
\subsection{Chevalley-Eilenberg Complex for Graded Lie Algebras}\label{sec_CECfGLA}
Consider a graded Lie algebra $\FR{g}$ equipped with the bracket of degree\footnote{We will only be interested in the Lie algebra of vector fields, whose bracket has degree 0, and Lie algebra of Hamiltonian functions, whose bracket has degree $-2$, but the formula given below is valid for general degrees.} $\nb$,
\bea
&&\deg [u,v]=\deg u+\deg v-\nb,\nn\\
&&[u,v]=-(-1)^{(|u|+\nb)(|v|+\nb)}[v,u],~~~u,v\in\FR{g}.\nn\eea
The Chevalley-Eilenberg (CE) chain complex is defined as $c_{\sbullet}=S^{\sbullet}(\FR{g}[\nb+1])$, i.e. the symmetric algebra of $\FR{g}$ with degree shifted by $\nb+1$ (sometimes, this shifting of degree is called suspension). Note that if $\FR{g}$ is of degree 0, then $S^{\sbullet}(\FR{g}[1])=\wedge^{\sbullet}\FR{g}$. An element of $c_{k}$ is written as
\bea (v_1,v_2,\cdots v_k)\in c_k,~~~~v_i\in\FR{g},\nn\eea
with the symmetry property
\bea (v_1,\cdots, v_i, v_{i+1}, \cdots, v_k) = (-1)^{(|v_i|+\nb+1)(|v_{i+1}|+\nb+1)}   (v_1,\cdots, v_{i+1}, v_{i}, \cdots, v_k) ~.\label{sym_prop}\eea
The boundary operator $\partial_I:~c_{\sbullet}\to c_{\sbullet-1}$ is defined as follows
\bea
\partial_I(v_1, \cdots,v_{k+1})& = &\sum_{i<j} (-1)^{s_{ij}} \big((-1)^{|v_i|}[v_i, v_j], v_1, \cdots , \hat{i}, \cdots , \hat{j}, \cdots , v_{k+1}\big)~,\nn\\
t_i&=&(|v_i|+\nb+1)(|v_1|+\cdots |v_{i-1}|+(i-1)(\nb+1)),\nn\\
s_{ij}&=&t_i+t_j-(|v_i|+\nb+1)(|v_j|+\nb+1).\label{CE_boundary}\eea
The sign factor $s_{ij}$ is called the Koszul sign; it is incurred by moving $v_i,v_j$ to the very front; $\hat i$ means skipping $v_i$.

Define the CE cochain complex $c^{\sbullet}$ as a multi-linear map from $c_{\sbullet}$ to a certain $\FR{g}$-module $\Mod$
\bea c^k(v_1,\cdots, v_k)\in \Mod,\nn\eea
The coboundary operator $\delta:\;c^{\sbullet}\to c^{\sbullet+1}$ is defined as follows
\bea
&&(\delta c^k)(v_1, \cdots,v_{k+1})=c^k\big(\partial_I(v_1, \cdots , v_{k+1})\big)\nn\\
&&\hspace{3.1cm}-\sum_{i} (-1)^{t_{i}+|v_i|\nb+(|v_i|+\nb)\deg c^k} v_i\circ c^k \big(v_1, \cdots , \hat{i}, \cdots , v_{k+1}\big)~,\label{CE_diff}\eea
where $\circ$ denotes $\FR{g}$-module action. The only condition on $\deg c^k$ from demanding $\delta^2=0$ is $\deg \delta c^k =\deg c^k+1$. In sec.\ref{sec_graph}, when the cochains are represented as graphs, $\deg c^k$ will be given in terms of the data of a graph
\bea \deg c^k=\nb E-(\nb+1)V,\label{deg_cochain}\eea
where $V$ and $E$ are the number of the vertices and edges of a graph. For now it is convenient to think of $\deg c^k$ naively as the degree carried by the symbol $c^k$.

What is more relevant for our application is the \emph{relative} CE complex. Let us quickly recall its definition. Let $K$ be the Lie-subgroup of the Lie group of $\FR{g}$, then a cochain $c^{\sbullet}$ is said to be horizontal w.r.t $K$ if $c^n(v_1,\cdots v_n)=0$ whenever any of the $v_i\in \FR{k}=\textrm{Lie}\,{}_K$. And $c^{\sbullet}$ is said to be invariant if it is invariant under the action of $K$. A cochain that is horizontal and invariant is said to be \emph{basic}. The relative CE complexes consist of cochains basic w.r.t $K$. In our case, the Lie group for $\FR{g}$ does not exist, rather it is the pseudo-group $\Gamma$, nevertheless the whole setup still makes sense.

We will also have need of the  \emph{cyclic bar complex} \cite{CycBar}.
Let $\A$ be a graded algebra, then at degree $k$, $B_k=\otimes^{k}\A[1]$ mod a relation of cyclic permutation. A typical element is commonly written as $[g_1|g_2|\cdots|g_k]\in B_{k}$, with the relation
\bea
[g_1|g_2|\cdots|g_k]\sim (-1)^{(|g_1|+1)(|g_2|+\cdots +|g_k|+k-1)}[g_2|\cdots |g_k|g_1].\label{cyc_sym_prop}\eea
The differential $\partial_H$ acts according to
\bea&&\partial_H[g_1|g_2|\cdots|g_k]=-\sum_{1\leq j\leq k}(-1)^{u_{j}+(|g_j|+1)}[g_jg_{j+1}|g_{j+2}|\cdots |g_k|g_1|\cdots|g_{j-1}],~~~~k+1\equiv1\nn\\
&&\hspace{2.1cm}u_j=\sum_{i=1}^{j-1}(|g_i|+1)\sum_{l=j}^k(|g_l|+1)~.\label{bar_diff}\eea
If $\FR{g}$ acts on $\A$, then we define an \emph{anti}-$\FR{g}$-module action on $B_k$ as
\bea &&u\circ[g_1|g_2|\cdots|g_k]=-\sum_{1\leq j\leq
k}(-1)^{w_{j}+|u|\nb+|u|+\nb}[g_1|\cdots |g_{j-1}|u\circ g_j|g_{j+1}|\cdots g_k],\label{g_action_B}\\
&&\hspace{2.15cm}w_{j}=(|u|+\nb)\sum_{l=1}^{j-1}(|g_l|+1),\nn\eea
satisfying $u\circ v\circ-(-1)^{(|u|+\nb)(|v|+\nb)}v\circ u\circ=-[u,v]\circ$.

The \emph{extended CE complex} is the tensor product of $c_{\sbullet}$ and $B_{\sbullet}$
\bea c_{p,q}=c_{p}\otimes B_q.\nn\eea
The total differential for the cochain complex $c^{p,q}$ is induced from $\partial_I$, $\partial_H$ and the $\FR{g}$-action Eq.\ref{g_action_B}, its full expression is relegated to the appendix.

\subsection{Lie Algebra of Vector Fields}\label{sec_LAoVF}
Consider several variants of the Lie algebra of tangent vector fields. Let $\M$ be an $N$-manifold, locally of type $\mathbbm{R}$, and $\E$ is a graded vector bundle over $\M$ with fibre $\BB{V}$.

1. Let $\FR{g}$ be the 'Lie algebra of formal $\Gamma_0$ vector-fields (see sec.2 of ref.\cite{BottHaefliger}), which can simply be thought of as vector fields on $\mathbbm{R}$ with formal power series as coefficients. Let $K$ be a subgroup of $GL(\mathbbm{R})$, in particular $K=SP(\mathbbm{R})$. We consider the relative CE complex $c^{\sbullet}(\FR{g},K)$ taking value in the real or complex number (trivial $\FR{g}$-module).

$1^\prime$. $\FR{g}=\textrm{vect}(\M)$, and the CE complex $\FR{c}^{\sbullet}$ taking value in functions of $\M$ with the obvious $\FR{g}$-module action.

2. Lie algebra $\FR{g}$ as in 1, and $\A=\textrm{End}(\BB{V})$, we consider the extended CE complex $c^{\sbullet,\sbullet}(\FR{g},K;\A)$ valued in $\BB{R}$ or $\BB{C}$, basic w.r.t $K$ and invariant under diagonal action of $\textrm{End}(\BB{V})$ on $\A$ (by conjugation).

$2^\prime$. $\FR{g}=\textrm{vect}(\M)$ and $\A=\textrm{End}(\E)$, we consider the extended CE complex $\FR{c}^{\sbullet,\sbullet}$ taking value in $C^{\infty}(\M)$.

The central result of this paper is the following prescription that turns the cocycles of type 1, 2 into cocycles of type $1^\prime$, $2^\prime$.  To keep the discussion uniform, we do not distinguish between Hamiltonian vector fields or Hamiltonian functions, one can think of the latter as a Lie algebra with Lie bracket of degree 0, or a Lie algebra with Poisson bracket of degree $-\nb$, with $\nb$ being the degree of the symplectic form.

Pick a base point $\SL{x}$, let $u\in\textrm{vect}(\M)$, and $\phi^{-1}_{\textsl{x}*}u\in\textrm{vect}(\mathbbm{R})$ be the push forward. Define\footnote{Hopefully, the hat here will not be confused with the hat in Eq.\ref{CE_diff}.}
\bea \hat{u}=\phi^{-1}_{\textsl{x}*}u-\iota_u\FR{G}.\label{def_hat}\eea
As this formula is quite central to our construction, let us briefly comment on it.
Note $\hat{u}$ always vanishes at the origin of $\mathbbm{R}$, so it takes value in the Lie algebra of $\Gamma_0$. Recall that $\FR{B}$ is a principal bundle with structure group $\Gamma_0$, under the local trivialization $\FR{B}\big|_U\sim \M\big|_U\times \Gamma_0$, we have the lift of $u$ as $u\to (u,\hat u)$. The following proposition says that this is a parallel lifting over $U$.
\begin{proposition}\label{prop_key_id}
\bea [\hat{u},\hat{v}]=\widehat{[u,v]}-u\circ\hat{v}+(-1)^{(|u|+\nb)(|v|+\nb)}v\circ\hat{u}\label{key_id},\eea
where $u\circ$ denotes the differentiation w.r.t the basis point
\bea u\circ =u^A(\textsl{x})\frac{\partial}{\partial \textsl{x}^A}.\label{vect_action}\eea
\end{proposition}

\begin{proof}: Clearly we have
\bea D_A\phi^B_{\textsl{x}}=\Big[\frac{\partial}{\partial\textsl{x}^A}-\FR{G}_B^A\frac{\partial}{\partial\xi^A}\Big]\phi_{\textsl{x}}^B=0\nn\eea
from the chain rule. Thus we have
\bea \partial_{A}(\phi^{-1}_{\textsl{x}*}u)-[\FR{G}_A,\phi^{-1}_{\textsl{x}*}u]=0.\label{temp4}\eea
This, in conjunction with the flatness of $\FR{G}$ given in Eq.\ref{flatness}, gives us the desired result\qed\end{proof}

Choose any \emph{global} section $\phi$ of $\FR{B}/K$, the exponential map to be introduced later is one possible choice, then prop.\ref{prop_key_id} leads to.
\begin{proposition}\label{prop_cochain_convert}
Pick any cochain $c^k$ of the type 1, the following cochain of type 1${}^\prime$ is well-defined
\bea \FR{c}^k(u_1,\cdots, u_k)=c^k(\hat{u}_1,\cdots,\hat{u}_k),~~~u_i\in\textrm{vect}(\M).\label{cochain_convert}\eea
i.e. the mapping $c^{\sbullet}$ to $\FR{c}^{\sbullet}$ is a morphism of complexes and $\FR{c}$ is valued in $C^{\infty}(\M)$.
\end{proposition}
\begin{proof}
In general the lhs of Eq.\ref{cochain_convert} only gives a function on $\FR{B}$. But if
$c^{\sbullet}$ is basic w.r.t the group $K$, the lhs of Eq.\ref{cochain_convert} gives a function on the quotient $\FR{B}/K$. Then through a global section of $\FR{B}/K$, one can pull back this function to $\M$. For similar discussions, see also sec.4 \cite{BottSegal}.

To show the mapping is a morphism
\bea
&&\delta c^k (\hat v_1, \cdots,\hat v_{k+1})\nn\\
& = &\sum_{i<j} (-1)^{s_{ij}} c^k \big((-1)^{|v_i|}[\hat v_i, \hat v_j], \hat v_1, \cdots , \hat{i}, \cdots , \hat{j}, \cdots , \hat v_{k+1}\big)~,\nn\\
& = &\sum_{i<j} (-1)^{s_{ij}} \FR{c}^k \big((-1)^{|v_i|}[v_i, v_j], v_1, \cdots , \hat{i}, \cdots , \hat{j}, \cdots , v_{k+1}\big)~,\nn\\
&&+\sum_{i<j} (-1)^{s_{ij}} c^k \big((-1)^{|v_i|}\big(-v_i\circ\hat{v}_j+(-1)^{(|v_i|+\nb)(|v_j|+\nb)}v_j\circ\hat{v}_i\big), \hat v_1, \cdots , \hat{i}, \cdots , \hat{j}, \cdots , \hat v_{k+1}\big)~,\nn\\
& = &\sum_{i<j} (-1)^{s_{ij}} \FR{c}^k \big((-1)^{|v_i|}[v_i, v_j], v_1, \cdots , \hat{i}, \cdots , \hat{j}, \cdots , v_{k+1}\big)~,\nn\\
&&-\sum_{i\neq j} (-1)^{t_{i}+|v_i|+(|v_i|+\nb)(|v_1|+\cdots |v_{j-1}|+(j-1)(\nb+1))} c^k \big(\hat v_1, \cdots , \hat{i}, \cdots , v^i\circ\hat v_j,\cdots, \hat v_{k+1}\big)\label{above}.\eea
where the symbols $t_i,\,s_{ij}$ were defined in Eq.\ref{CE_boundary}. Note that the last term of Eq.\ref{above} can be written so regardless of whether $i<j$ or not, because $(|v_i|+\nb)(|v_i|+\nb+1)=0\textrm{mod}2$.

Thus if we set the $\FR{g}$-module action to be
\bea &&(u\circ \FR{c}^k)(v_1,\cdots,v_k)=(-1)^{(|u|+\nb)\deg c^k}\times\nn\\
&&\hspace{2.3cm}\sum_{i}(-1)^{(|u|+\nb)(|v_1|+\cdots |v_{i-1}|+(i-1)(\nb+1))}c^k(\hat v_1,\cdots, (-1)^{(\nb+1)(|u|+\nb)}u\circ \hat v_i ,\cdots, \hat v_k),\label{module_action1}\eea
The last term of Eq.\ref{above} is written as
\bea -\sum_{i} (-1)^{t_{i}+\nb|v_i|+(|v_i|+\nb)\deg c^k}v_i\circ\FR{c}^k \big(v_1, \cdots , \hat{i}, \cdots , v_{k+1}\big)\nn,\eea
in agreement with Eq.\ref{CE_diff}. Thus the rhs of Eq.\ref{above} is the differential of a cochain $\FR{c}^k$ valued in $C^{\infty}(\M)$ with action defined in Eq.\ref{module_action1}.
Note that in this equation, the sign $(-1)^{(\nb+1)(|u|+\nb)}$ can be understood as resulting from commuting the action of $u$ (of degree $|u|+\nb$) across the suspension (of degree $\nb+1$)\qed\end{proof}
\begin{remark}\label{one_remark}
In the end of section \ref{sec_CD}, the importance of $c^{\sbullet}$ being basic w.r.t $K$ will be understood at a much more concrete level.
\end{remark}

Next we investigate how does $\hat u$ respond to a change of the global section $\tilde \phi^{-1}_{\textsl{x}}=h_{\textsl{x}}\circ \phi^{-1}_{\textsl{x}}$, where $h_{\SL{x}}\in\Gamma_0$, i.e. a diffeomorphism of $\mathbbm{R}$: $\xi\to h_{\SL{x}}(\xi)$ preserving the origin.
From Eq.\ref{gauge_trans_gen}, we read off (the sign flip as well as the change $h\to h^{-1}$ has been explained there too)
\bea\hat u \stackrel{h_{\SL{x}}}{\to} h_{\SL{x}*}\phi^{-1}_{\SL{x}*}u+\iota_u\big(h^{-1}_{\SL{x}}dh_{\SL{x}}-h^{-1}_{\SL{x}}\FR{G}h_{\SL{x}}\big).\nn\eea
Letting $h$ be infinitesimal and generated by $\Psi\in\textrm{vect}(\mathbbm{R})$, the change in $\hat u$ is
\bea \delta_{\Psi}\hat u\big|_{h_{\SL{x}}(\xi)}=\big(h_{\SL{x}*}\phi^{-1}_{\SL{x}*}u-\phi^{-1}_{\SL{x}*}u\big|_{h_{\SL{x}}(\xi)}\big)+\iota_u\big(d\Psi-[\FR{G},\Psi]\big).\nn\eea
In this limit, the first round brace constitutes the definition of a Lie derivative and hence gives
$-L_{\Psi}\phi^{-1}_{\SL{x}*}u=-[\Psi,\phi^{-1}_{\SL{x}*}u]$, leading to
\begin{proposition}
Under an infinitesimal change of local isomorphism as above, $\hat u$ changes according to
\bea \delta_{\Psi}\hat u=[\hat u,\Psi]+u\circ\Psi\label{change_hat_u},\eea
where $u\circ$ is as in Eq.\ref{vect_action}.
\end{proposition}
If we pick $c^k$ to be a cocycle, then from prop.\ref{prop_cochain_convert} $\FR{c}^k$ is also a cocycle. Now we can show that the cohomology class of $\FR{c}^k$ is independent of the choice of the section $\phi$.
\begin{proposition}\label{prop_connection_ind}
Under an infinitesimal change of the section $\phi^{-1}_{\SL{x}}\to h_{\SL{x}}\circ\phi^{-1}_{\SL{x}}$, where $h$ is generated by $\Psi\in \textrm{vect}(\mathbbm{R})$ fixing the origin. And pick $c^k$-closed,
then the class defined by
\bea \FR{c}^k(u_1,\cdots, u_k)=c^k(\hat{u}_1,\cdots,\hat{u}_k),~~~u_i\in\textrm{vect}(\M)\nn\eea
is invariant.
\end{proposition}
For clarity, we prove the proposition for a smooth manifold and $k=2$, which is more than sufficient to make clear the idea.
\begin{proof}
Define $c_{\Psi}^1(\hat u)=c^2(\hat u,\Psi)$, and using that $c^2$ is closed, we can compute
\bea 0=\delta c^2(\hat u,\hat v,\Psi)=c^2([\hat u,\hat v],\Psi)+c^2(\hat v,[\hat u,\Psi])-c^2(\hat u,[\hat v,\Psi]).\label{above1}\eea
Use Eq.\ref{key_id} to convert the first term of Eq.\ref{above1}
\bea c^2([\hat u,\hat v],\Psi)=c^2(\widehat{[u,v]},\Psi)-c^2(u\circ\hat v,\Psi)+c^2(v\circ\hat u,\Psi).\nn\eea
Add and subtract on the rhs the following terms $-c^2(\hat v,u\circ\Psi)+c^2(\hat u,v\circ\Psi)$, we get
\bea c^2([\hat u,\hat v],\Psi)=\delta\FR{c}^1_{\Psi}(u,v)+c^2(\hat v,u\circ\Psi)-c^2(\hat u,v\circ\Psi),\nn\eea
and for rhs of Eq.\ref{above1},
\bea 0=\delta\FR{c}^1_{\Psi}(u,v)+c^2(\hat v,u\circ\Psi+[\hat u,\Psi])-c^2(\hat u,v\circ\Psi+[\hat v,\Psi])
=\delta\FR{c}^1_{\Psi}(u,v)-c^2(\delta_{\Psi}\hat u,\hat v)-c^2(\hat u,\delta_{\Psi}\hat v).\nn\eea
Thus we can conclude
\bea\delta_{\Psi}\FR{c}^2(u, v)=\delta\FR{c}^1_{\Psi}(u,v),\nn\eea
thus the cohomology class is unaltered \qed
\end{proof}

Finally, we include the representation of a $Q$-structure introduced in Eq.\ref{def_rep}. Since $Q^{\uparrow}$ is a globally defined vector field on $\E$, one can define similarly the quantity
\bea \widehat T=\widehat{Q^{\uparrow}}-\widehat Q,\label{hat_T}\eea
and call it the \emph{representation} matrix of $Q$. Note that on the rhs $\widehat Q \in \textrm{vect}(\mathbbm{R})$ is regarded as a vector field of $\mathbbm{R}\times \BB{V}$ trivially through the direct product structure.

Since both $\widehat Q$ and $\widehat{Q^\uparrow}$ satisfy a similar relation derived from Eq.\ref{key_id},
\bea (\widehat{Q^{\uparrow}})^2=-Q^{\uparrow}\circ \widehat{Q^{\uparrow}},~~~~\widehat{Q}^2=-Q\circ \widehat{Q},\label{MC}\eea
we can derive
\bea \widehat T^2=(\widehat{Q^{\uparrow}}-\widehat Q)^2=-[\widehat Q,\widehat{Q^{\uparrow}}-\widehat Q]+(\widehat{Q^{\uparrow}})^2-\widehat{Q}^2=-[\widehat Q,\widehat T]-Q\circ\widehat Q+Q^{\uparrow}\circ \widehat{Q^{\uparrow}}.\label{above2}\eea
And we note that $Q^{\uparrow}=Q$ when restricted to $\M$, as $Q^{\uparrow}$ is a lift of $Q$ and the fibre dependence of $Q^{\uparrow}$ is linear. One can evaluate Eq.\ref{above2} at $z_{\ga}=0$ (not $\zeta_{\ga}=0$!, see the end of sec.\ref{sec_BHC} for notations)
\bea \widehat T^2\Big|_{\M}=-\big([\widehat Q,\widehat T]+Q\circ \widehat{T}\big)~\Big|_{\M}.\label{MC_adv}\eea
Besides, using Eq.\ref{change_hat_u}, we conclude that $\widehat T$ transforms
\bea \delta_{\epsilon}\widehat T^{\gb}_{~\ga}= \widehat T^{\beta}_{~\gamma}\epsilon^{\gamma}_{~\alpha}-(-1)^{\beta+\gamma}\epsilon^{\beta}_{~\gamma}\widehat T^{\gamma}_{~\alpha},\label{gauge_transform_improve}\eea
in contrast to Eq.\ref{gauge_transform}. Thus
\bea \widehat T \in \Gamma(\textrm{End}\,\E).\nn\eea

\section{Exponential Maps}\label{sec_EM}
In this section we give some concrete constructions using the exponential map.

First, we construct a section $\phi^{-1}_x:~\M\to \FR{B}$ (recall from sec.\ref{sec_BHC}).
By the folk-theorem for graded and super manifolds (for supermanifold case proved independently by Batchelor \cite{MR536951}, by Berezin
 \cite{berezin}, by  Gawedzki \cite{MR0489701} and with straightforward generalization to the graded case),
  one can always split $\M$ non-canonically into a direct sum of graded vector bundles
\bea \M\sim E_1[1]\oplus E_2[2]\oplus\cdots,\label{splitting}\eea
where each $E_i$ is over $M=|\M|$. For each vector bundle we choose a connection $A_{\mu}$, and also choose a connection $\Gamma^{\mu}_{\nu\rho}$ for $TM$. For convenience, we assume $\Gamma$ torsionless $\Gamma^{\mu}_{[\nu\rho]}=0$. We build a local isomorphism Eq.\ref{local_iso} by first building a local isomorphism $M\supset U\to \BB{R}^n$ through certain exponential map and then attach the graded vector bundle structure using parallel-transport.

Consider the body $M=|\M|$ first.
The isomorphism $U\to \BB{R}^n$ can be done in many ways. One way is to
use a \emph{geodesic exponential map}.
\bea \phi_x^{\mu}=x^{\mu}+\xi^{\mu}-\frac12\Gamma^{\mu}_{\ga\gb}\xi^{\ga}\xi^{\gb}+\Big(-\frac16\partial_{\gc}\Gamma^{\mu}_{\ga\gb}+\frac13\Gamma^{\mu}_{\kappa\gc}\Gamma^{\kappa}_{\ga\gb}\Big) \xi^{\gc}\xi^{\ga}\xi^{\gb}+{\cal O}(\xi^4),\label{exponent_plain}\eea
where all the $\Gamma$ are evaluated at $x$.

Alternatively one can choose a metric $g_{\mu\nu}$ for $|\M|$, and locally fixes a vielbein-a set of $n$ orthonormal basis for $\Gamma(TM)$, $e^{\mu}_a  \partial_{\mu}$ with $g_{\mu\nu}e^{\mu}_ae^{\nu}_b=\delta_{ab}$. The vielbeins are patched together with structure group $SO(n)$ instead of $GL(n)$
\bea \tilde e^{\mu}_a(y)\frac{\partial x^{\nu}}{\partial y^{\mu}}=e^{\nu}_b(x)U^b_{~a},~~U^a_{~b}\in SO(n).\nn\eea
Now one can form a different exponential map by using $e^{\mu}_a$ as the velocity-field for the flow equation, leading to
\bea \phi_{so}^{\mu}=x^{\mu}+\xi^ae^{\mu}_a+\frac12\big(e_b^{\rho}\partial_{\rho}e^{\mu}_a\big)\xi^{a}\xi^{b}+\frac16\big(e^{\sigma}_c\partial_{\sigma}(e^{\rho}_b\partial_{\rho}e^{\mu}_a))\big) \xi^{c}\xi^{a}\xi^{b}+{\cal O}(\xi^4).\label{exponent_ortho}\eea
where as before the $e^{\mu}_a$ are evaluated at ${x}$.

The following discussion up to prop.\ref{prop_global_section_sp} is well-known to the experts, yet we go through it carefully. At first sight, the exponential map $\phi_x$ does not give a section $M\to\FR{B}$, because this map not only depends on the base point $x$ around which we do the expansion, but also on the choice of the coordinate system in the neighbourhood of $x$. Namely $\phi_x$ is not quite invariant under local diffeomorphisms of $M$.
To investigate how $\phi_x$ reacts to a diffeomorphism, let $\tilde{{x}}^{\mu}={x}^{\mu}-v^{\mu}$ be infinitesimal diffeomorphism, where $v^{\mu}$ is a vector field vanishing at $x$, we check up to order $\xi^3$ that Eq.\ref{exponent_plain} satisfies
\bea \phi^{\mu}_{\tilde{{x}}}(\tilde\xi)= \phi^{\mu}_{{x}}(\xi)+{\cal O}(v^2),\label{equiv}\eea
where $\tilde \xi^{\mu}=\xi^{\mu}-\xi^{\rho}\nabla_{\rho}v^{\mu}$ is an infinitesimal $GL$ rotation. The relation Eq.\ref{equiv} and the formula for $\tilde{\xi}$ is in fact exact to all orders in $\xi$, which can be understood as follows: once the connection is fixed, the flow that defines the exponential map is uniquely fixed by the initial velocity $\xi^{\mu}$, thus the flow computed in the $x$-coordinate system equals the flow computed in the $\tilde{{x}}$-coordinate with initial velocity $(\partial\tilde{{x}}^{\mu}/\partial{ x}^{\nu})\xi^{\nu}$ (of course, one has to transform the connection to the $\tilde x$ coordinates too). Hence the exponential map $\phi_x$ is \emph{not} a section of $\FR{B}$, but the effect of a local diffeomorphism is merely a $GL$ rotation of the $\xi$'s. To summarize the above discussion, we have the
\begin{proposition}\label{prop_global_section}
The exponential map Eq.\ref{exponent_plain} gives a global section of the bundle $\FR{B}/\textrm{GL}$ over $M$, similarly Eq.\ref{exponent_ortho} gives a global section of $\FR{B}/\textrm{SO}$
\end{proposition}
\begin{proof}Once the connection is chosen, Eq.\ref{equiv} shows that $\phi_x$ is (up to a $GL$ rotation) independent of the coordinate system, thus it only depends on the point $x$\qed
\end{proof}

For later use, consider $M$ with a symplectic structure $\Omega$, one can always pick a torsionless connection that preserves $\Omega$. Then one can demand that the exponential map preserves $\Omega$, the required map is
\bea &&\phi_{sp}^{\mu}={x}^{\mu}+\xi^{\mu}-\frac12\Gamma^{\mu}_{\ga\gb}\xi^{\ga}\xi^{\gb}+\Big\{-\frac16\partial_{\gc}\Gamma^{\mu}_{\ga\gb}+\frac13\Gamma^{\mu}_{\kappa\gc}\Gamma^{\kappa}_{\ga\gb}
-\frac{1}{24}\textsf{R}^{\mu}_{~\gc\ga\gb}\Big\} \xi^{\gc}\xi^{\ga}\xi^{\gb}+{\cal O}(\xi^4),\label{exponent_shifted}\\
&&\hspace{3.5cm}\textsf{R}^{\mu}_{~\gc\ga\gb}=(\Omega^{-1})^{\nu\mu}R_{\nu\gc~\gb}^{~~\lambda}\Omega_{\lambda\ga},\nn\eea
where the quantities $\textsf{R}$, $\Omega$ etc are evaluated at the point $\SL{x}$. It may be checked that the induced symplectic form in $\BB{R}^{2n}$ is $\Omega(\SL{x})$. The map $\phi_{sp}$ has the same equivariance property as Eq.\ref{equiv} and by picking a vielbein similarly as above, we easily show
\begin{proposition}\label{prop_global_section_sp}
$\phi_{sp}$ is a section of the bundle $\FR{B}/SP(n)$.
\end{proposition}

Next we bring in the rest of the graded structure, as a matter of notation, let the underlined indices $\udl A$ denote all coordinates other than that of the body $x^{\mu}$, and we use the typewriter font to denote the combination
\bea {\tt x}^{\udl A}={x}^{\udl A}+\xi^{\udl A}.\label{notation}\eea
The exponential map for the body is as Eq.\ref{exponent_plain}, while for the graded sector, it is
\bea
\phi^{\udl A}&=&{\tt x}^{\udl A}-\xi^{\ga}A^{\udl A}_{\ga \udl B}{\tt x}^{\udl B}+\frac12\xi^{\ga}\xi^{\gb}\Big(A_{\ga}A_{\gb}+\Gamma^{\rho}_{\ga\gb}A_{\rho}-\partial_{\ga}A_{\gb}\Big)^{\udl A}_{~\udl B}{\tt x}^{\udl B}
+\frac{1}{6}\xi^{\ga}\xi^{\gb}\xi^{\gc}\Big(\partial_{\gc}\big(A_{\ga}A_{\gb}+A_{\rho}\Gamma^{\rho}_{\ga\gb}-\partial_{\ga}A_{\gb}\big)\nn\\
&&-\Gamma^{\rho}_{\ga\gb}\big(A_{(\rho}A_{\gc)}+2A_{\sigma}\Gamma^{\sigma}_{\rho\gc}-\partial_{(\rho}A_{\gc)}\big)
-\big(A_{\ga}A_{\gb}+A_{\rho}\Gamma^{\rho}_{\ga\gb}-\partial_{\ga}A_{\gb}\big)A_{\gc}\Big)^{\udl A}_{~\udl{B}}{\tt x}^{\udl{B}}+{\cal O}((\xi^{\mu})^4),\label{exponent_graded}\eea
Notice that since the connection $A$ acts within each $E_i$, we always have $\deg\, A^A_{\mu\,B}=0$.

The Gronthendieck connection defined in Eq.\ref{Gro_conn} is given by the formula
\bea \FR{G}=\left[\FR{G}^{\mu},~~\FR{G}^{\udl A}\right]=\left[dx^{\nu},~~ dx^{\udl B}\right]\left[\begin{array}{cc}
         \frac{\partial\phi^{\rho}}{\partial x^{\nu}} & \frac{\partial\phi^{\udl C}}{\partial x^{\nu}} \\
         \frac{\partial\phi^{\rho}}{\partial x^{\udl B}} & \frac{\partial\phi^{\udl C}}{\partial x^{\udl B}}
       \end{array}\right]
\left[\begin{array}{cc}
         \frac{\partial\phi^{\mu}}{\partial\xi^{\rho}} & \frac{\partial\phi^{\udl A}}{\partial \xi^{\rho}} \\
         \frac{\partial\phi^{\mu}}{\partial\xi^{\udl C}} & \frac{\partial\phi^{\udl A}}{\partial \xi^{\udl C}}
       \end{array}\right]^{-1},\nn\eea
which is calculated to be
\bea &&\mathfrak{G}=dx^{\gc}\big(\delta^{\mu}_{\gc}+\Gamma^{\mu}_{\gb\gc}\xi^{\gb}-\frac13R_{\gc\ga~\gb}^{~~\;\mu}\xi^{\ga}\xi^{\gb}\big)\frac{\partial}{\partial\xi^{\mu}}\nn\\
&&\hspace{2.5cm}+dx^{\udl A}\frac{\partial}{\partial {\xi}^{\udl A}}+dx^{\gc}{\tt x}^{\udl C}\big(A_{\gc}+\frac12\xi^{\gb}F_{\gb\gc}-\frac16\xi^{\ga}\xi^{\gb}\BS{\nabla}_{\ga}F_{\gc\gb}\big)^{\udl A}_{~\udl C}\frac{\partial}{\partial {\xi}^{\udl A}}+{\cal O}(\xi^3),\label{summarize}\eea
where $\BS{\nabla}$ is the covariant derivative with connection $\Gamma+A$ and $F_{\ga\gb}$ is the curvature for $A$. If one uses the map Eq.\ref{exponent_shifted}, then the first line of $\FR{G}$ becomes
\bea \FR{G}_{sp}=dx^{\gc}\big(\delta_{\gc}^{\mu}+\Gamma^{\mu}_{\gb\gc}\xi^{\gb}-\big(\frac18{\SF R}_{\gc~{\ga\gb}}^{~{\mu}}
+\frac14{\SF R}_{\gc\ga~\gb}^{~~\,{\mu}}\big)\xi^{\ga}\xi^{\gb}\big)\frac{\partial}{\partial\xi^{\mu}}+\cdots,\label{Gro_conn_shifted}\eea
in particular $\FR{G}$ has a Hamiltonian lift since it is $\FR{sp}$-valued.

As an example, we take $\M=T[1]M$ with body coordinates $x^{\mu}$ of deg 0 and $\deg~1$ coordinates $v^{\mu}$, as well as their flat space counterparts $\xi^{\mu},\;\nu^{\mu}$, and ${\tt v}=v+\nu$ as above. For the homological vector field $Q=v^{\mu}\partial_{\mu}$, we can compute\footnote{As a note on computation, it is much easier to compute $\widehat Q$ directly than to compute $\phi^{-1}_*Q$ and $\iota_{Q}\FR{G}$ separately.}
\bea {\widehat Q}=\frac 12{\tt v}^{\mu}\xi^{\ga}\xi^{\gb}R_{\mu\ga~\gb}^{~~\;\rho}\frac{\partial}{\partial \xi^{\mu}}+{\tt v}^{\mu}{\tt v}^{\kappa}\Big(\xi^{\ga}R_{\mu\ga~\kappa}^{~~\;\lambda}+\frac12\xi^{\ga}\xi^{\gb}\nabla_{\ga}R_{\mu\gb~\kappa}^{~~\;\lambda}\Big)\frac{\partial}{\partial {\nu}^{\lambda}}+{\cal O}(\xi^3).\nn\eea
We will attach some bundle structures to get a representation of $Q$ in sec.\ref{sec_WfGViQc}.

\subsection{Connection Dependence}\label{sec_CD}
In obtaining the expression Eq.\ref{summarize}, we have made some non-canonical choices, such as the splitting in Eq.\ref{splitting} as well as the choices of $\Gamma^{\mu}_{\nu\rho}$, $A_{\mu}$. The resulting change of the Gronthendieck connection is predicted by the general formula Eq.\ref{gauge_trans_gen}, since all of the choices above merely serve to define the local isomorphism Eq.\ref{local_iso}. However, we would like to check Eq.\ref{gauge_trans_gen} explicitly, which also serves to strengthen the credibility of Eq.\ref{summarize}.

Under a change of the connection for $TM$, we write $\delta\Gamma^{\mu}_{\ga\gb}=\gamma^{\mu}_{\ga\gb}$, and we continue to assume $\gamma^{\mu}_{[\ga\gb]}=0$. Let
\bea\Psi=\Big(\frac{1}{2}\gamma^{\mu}_{\ga\gb}\xi^{\ga}\xi^{\gb}+\frac16(\nabla_{\ga}\gamma^{\mu}_{\gb\gc})\xi^{\ga}\xi^{\gb}\xi^{\gc}
+\frac{1}{12}\gamma^{\mu}_{\ga\gb}F_{\mu\gc}\xi^{\ga}\xi^{\gb}\xi^{\gc}\Big)\frac{\partial}{\partial\xi^{\mu}},\nn\eea
and we can check that
\bea[-d+\FR{G},\Psi]=\delta_{\Gamma}\FR{G},\nn\eea
which is the infinitesimal version of Eq.\ref{gauge_trans_gen}. The sign discrepancy has been explained earlier.

Likewise we can vary $A$, $\delta A_{\mu}=\vara_{\mu}$, and let
\bea \Psi=\vara_{\ga}\xi^{\ga}+\frac12(\BS{\nabla}_{\ga}\vara_{\gb})\xi^{\ga}\xi^{\gb}+\frac16(\BS{\nabla}_{\ga}\BS{\nabla}_{\gb}\vara_{\gc})\xi^{\ga}\xi^{\gb}\xi^{\gc},\nn\eea
it may be checked that
\bea [-d+\FR{G},\Psi]=\delta_{A}\FR{G}.\nn\eea

To check Eq.\ref{change_hat_u}, we use the geodesic exponential map Eq.\ref{exponent_plain} as an example, and compute the push-forward of a vector field
\bea (\phi^{-1}_{\SL{x}})_*u=\Big(u^{\mu}+\xi^{\ga}\nabla_{\ga}u^{\mu}+\xi^{\ga}\xi^{\gb}\big(\frac16u^{\rho}R_{\rho\ga~\gb}^{~~\,\mu}+\frac12\nabla_{\ga}\nabla_{\gb}u^{\mu}\big)\Big)
\frac{\partial}{\partial \xi^{\mu}}+{\cal O}(\xi^3).\nn\eea
Thus
\bea \hat u^{\mu}=\xi^{\ga}\partial_{\ga}u^{\mu}+\frac12\xi^{\ga}\xi^{\gb}\big(u^{\rho}R_{\rho\ga~\gb}^{~~\,\mu}+\nabla_{\ga}\nabla_{\gb}u^{\mu}\big),\label{hat_u}\eea
and if one changes the connection $\delta\Gamma^{\mu}_{\nu\rho}=\gamma^{\mu}_{\nu\rho}$
\bea \delta_{\Gamma}\hat u^{\mu}=\frac12\xi^{\ga}\xi^{\gb}\big(2\gamma^{\mu}_{\ga\gc}\nabla_{\gb}u^{\gc}-\gamma^{\gc}_{\ga\gb}\nabla_{\gc}u^{\mu}+u^{\gc}  \nabla_{\gc}\gamma^{\mu}_{\ga\gb}\big)=u\circ \Psi^{\mu}+[\hat u,\Psi]^{\mu}.\nn\eea
This is the verification of the assertion Eq.\ref{change_hat_u} in the case when the change of local isomorphism is caused by a change of the connection.

A change of the splitting Eq.\ref{splitting} corresponds to a shift
\bea  x^{\udl A}=x^{\udl A}+\epsilon^{\udl A}_{~\udl B} x^{\udl B},\nn\eea
the verification of Eq.\ref{equiv} is almost trivial given that the exponential map depends linearly on the combination ${\tt x}^{\udl A}={x}^{\udl A}+\xi^{\udl A}$.

\begin{remark}
We notice that in the expression Eq.\ref{hat_u} for $\hat u$, only the first term is non-covariant, and it takes value in the lie algebra $\FR{gl}$ (or $\FR{so}$ and $\FR{sp}$ if one uses the exponential map Eq.\ref{exponent_ortho} and Eq.\ref{exponent_shifted}). Thus if the CE cochain into which we plug $\hat u$ is basic w.r.t $GL$ (resp. $SO$, $SP$), the result is manifestly covariant. This is a more elementary way of understanding the necessitiy of basic-ness promised on page.\pageref{one_remark}.
\end{remark}

So far, we have checked explicitly all claims made in sec.\ref{sec_BHC} and \ref{sec_LAoVF}.

\section{Graph Complex}\label{sec_graph}
A cochain in the CE complex of the Lie algebra of formal vector fields or Hamiltonian functions can be presented as a graph, and the cohomology of the CE complex is subsequently computed as the cohomology of the graph complex. Unfortunately, the CE cohomology of vector fields are quite banal, see ref.\cite{Fuks}, while the same cohomology for the Hamiltonian functions is much more interesting, so is its associated graph complex.

As such we will be mostly studying the Lie algebra of Hamiltonian functions, which has a bracket of degree $\nb=0,-2,\cdots$, and the associated graph complex
is what we call the \emph{plain} graph complex. The key is the correspondence and its extension
\bea &&\hspace{1.9cm}\textrm{Chevalley-Eilenberg complex}~\mathop{\rightleftharpoons}_{\beta^{\dagger}}^{\beta}~\textrm{Graph complex},\\
&&\textrm{CE complex valued in cyclic bar complex}~\mathop{\rightleftharpoons}_{\beta^{\dagger}}^{\beta}~\textrm{Extended graph complex}.\label{CE_Gph_correspondence}\eea
Thus one can construct CE \emph{co}cycles from graph \emph{co}cycles and conversely construct graph cycles from CE cycles.
For more details, see the lecture notes \cite{BV_lecture}.

A graph is a finite 1-dimensional CW complex, in simpler words, it is a collection of vertices (0-cells) connected by edges (1-cells). The vertices are labelled  $1,2,...$ and each edge is oriented. The graph complex is the formal linear combination of graphs mod the relation: flipping the orientation of an edge or exchanging the labelling of two vertices flips the sign of the graph.

For weight systems of knots, we will also need the extended graph complex, which includes an extra oriented circle in addition to the above ingredients. Vertices can be placed on the circle (the peripheral ones) or away from the circle (internal ones); these two types of vertices are labelled separately. The edges can now also run between the two types of vertices. One can still flip the orientation of an edge at the cost of a minus sign. The peripheral vertices can also be cyclically permuted with the appropriate sign.

\begin{definition}
A particular type of the extended graph complex is the so called \emph{chord diagrams}, consisting of graphs that have tri-valent internal vertices and uni-valent peripheral vertices.
\end{definition}
These diagrams naturally appear in the finite type Vassiliev knot invariants \cite{Vassiliev}, and their weights is what we aim to construct eventually.

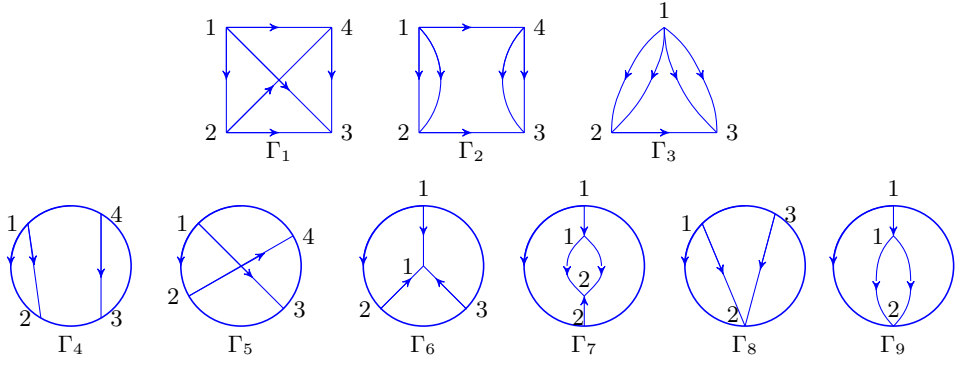
\begin{figure}[h]
\begin{center}
\begin{tikzpicture}[scale=.7]
\draw [-,blue] (0,0cm) -- (0,2cm)
      [-,blue] (0,0cm) -- (2,0cm)
      [-,blue] (2,0cm) -- (2,2cm)
      [-,blue] (0,2cm) -- (2,2cm);
\draw [->,blue] (0,2cm) -- (0,1cm);
\draw [->,blue] (2,2cm) -- (2,1cm);
\draw [->,blue] (0,0cm) -- (1,0cm);
\draw [->,blue] (0,2cm) -- (1,2cm);
\draw [-,blue](0,0cm)--(2,2cm);
\draw [->,blue](0,0cm)--(.9,.9cm);
\draw [-,blue](0,2cm)--(2,0cm);
\draw [->,blue](0,2cm)--(1.2,.8cm);
\node (p) at (0,0) [left] {\small $2$};
\node (q) at (2,0) [right] {\small $3$};
\node (s) at (0,2) [left] {\small $1$};
\node (r) at (2,2) [right] {\small $4$};
\node (g) at (1,0) [below] {\small $\Gamma_1$};
\end{tikzpicture}~~
\begin{tikzpicture}[scale=.7]
\draw [-,blue] (0,0cm) -- (0,2cm)
      [-,blue] (0,0cm) -- (2,0cm)
      [-,blue] (2,0cm) -- (2,2cm)
      [-,blue] (0,2cm) -- (2,2cm);
\draw [->,blue] (0,2cm) -- (0,1cm);
\draw [->,blue] (2,2cm) -- (2,1cm);
\draw [->,blue] (0,0cm) -- (1,0cm);
\draw [->,blue] (0,2cm) -- (1,2cm);
\draw [-,blue](0,0cm) arc (-45:45:1.414cm);
\draw [-,blue](2,2cm) arc (135:225:1.414cm);
\draw [->,blue](0,2cm) arc (45:0:1.414cm);
\draw [->,blue](2,2cm) arc (135:180:1.414cm);
\node (p) at (0,0) [left] {\small $2$};
\node (q) at (2,0) [right] {\small $3$};
\node (s) at (0,2) [left] {\small $1$};
\node (r) at (2,2) [right] {\small $4$};
\node (g) at (1,0) [below] {\small $\Gamma_2$};
\end{tikzpicture}~~
\begin{tikzpicture}[scale=.7]
\draw [-,blue] (0+3.5,0cm) -- (2+3.5,0cm);
\draw [->,blue] (0+3.5,0cm) -- (1.1+3.5,0cm);
\draw [->,blue] (1+3.5,2) to [out=-88,in=65] (.75+3.5,1);
\draw [-,blue] (.75+3.5,1) to [out=-115,in=45] (3.5,0);
\draw [->,blue] (1+3.5,2) to [out=-90-26.6*2,in=65] (.25+3.5,1);
\draw [-,blue] (.25+3.5,1) to [out=-115,in=90] (3.5,0);
\draw [->,blue] (1+3.5,2) to [out=-88,in=115] (1.25+3.5,1);
\draw [-,blue] (1.25+3.5,1) to [out=-65,in=135] (2+3.5,0);
\draw [->,blue] (1+3.5,2) to [out=-90+26.6*2,in=180-65] (1.75+3.5,1);
\draw [-,blue] (1.75+3.5,1) to [out=-65,in=90] (2+3.5,0);
\node (p1) at (3.5,0) [left] {\small $2$};
\node (q1) at (2+3.5,0) [right] {\small $3$};
\node (r1) at (1+3.5,2) [above] {\small $1$};
\node (g) at (1+3.5,0) [below] {\small $\Gamma_3$};
\end{tikzpicture}~~\\
\begin{tikzpicture}[scale=.8]
\draw [semithick,blue] (0,0) circle (1);
\draw [->,blue] (0,1) arc (90:180:1cm);
\draw [-,blue](-.717,.717cm)--(-.5,-.866cm);
\draw [->,blue](-.717,.717cm)--(-.717/2-.5/2,.717/2-.866/2cm);
\draw [-,blue](.5,.866cm)--(.5,-.866cm);
\draw [->,blue](.5,.866cm)--(.5,-.2cm);
\node (p) at (-.717,.717cm) [left] {\small $1$};
\node (q) at (-.5,-.866cm) [left] {\small $2$};
\node (s) at (.5,.866cm) [right] {\small $4$};
\node (r) at (.5,-.866cm) [right] {\small $3$};
\node (g) at (0,-1) [below] {\small $\Gamma_4$};
\end{tikzpicture}~~
\begin{tikzpicture}[scale=.8]
\draw [semithick,blue] (0,0) circle (1);
\draw [->,blue] (0,1) arc (90:180:1cm);
\draw [-,blue](-.717,.717cm)--(.717,-.717cm);
\draw [->,blue](-.717,.717cm)--(.2,-.2cm);
\draw [-,blue](-.866,-.5cm)--(.866,.5cm);
\draw [->,blue](-.866,-.5cm)--(0.4,0.217cm);
\node (p) at (-.717,.717cm) [left] {\small $1$};
\node (q) at (-.866,-.5cm) [left] {\small $2$};
\node (r) at (.717,-.717cm) [right] {\small $3$};
\node (s) at (.866,.5cm) [right] {\small $4$};
\node (g) at (0,-1) [below] {\small $\Gamma_5$};
\end{tikzpicture}~~
\begin{tikzpicture}[scale=.8]
\draw [semithick,blue] (0,0) circle (1);
\draw [->,blue] (0,1) arc (90:180:1cm);
\draw [-,blue](-.717,-.717cm)--(0,0cm);
\draw [->,blue](-.717,-.717cm)--(-.2,-.2cm);
\draw [-,blue](.717,-.717cm)--(0,0cm);
\draw [->,blue](.717,-.717cm)--(.2,-.2cm);
\draw [-,blue](0,1cm)--(0,0cm);
\draw [->,blue](0,1cm)--(0,0.5cm);
\node (p) at (0,0cm) [left] {\small $1$};
\node (q) at (-.717,-.717cm) [left] {\small $2$};
\node (r) at (.717,-.717cm) [right] {\small $3$};
\node (s) at (0,1cm) [above] {\small $1$};
\node (g) at (0,-1) [below] {\small $\Gamma_6$};
\end{tikzpicture}~~
\begin{tikzpicture}[scale=.8]
\draw [semithick,blue] (0,0) circle (1);
\draw [->,blue] (0,1) arc (90:180:1cm);
\draw [->,blue](0,-1cm)--(0,-.5cm);
\draw [->,blue](0,1cm)--(0,.5cm);
\draw [-<,blue](0,-.5cm) to [out=135,in=-90] (-2/7,0cm);
\draw [-,blue](-2/7,0cm) to [out=90,in=-135] (0,0.5cm);
\draw [-<,blue](0,-.5cm) to [out=45,in=-90] (2/7,0cm);
\draw [-,blue](2/7,0cm) to [out=90,in=-45] (0,0.5cm);
\node (p) at (0,0.5cm) [left] {\small $1$};
\node (q) at (0,-.5cm) [above] {\small $2$};
\node (r) at (0,1cm) [above] {\small $1$};
\node (s) at (-.1,-.85cm) [ ] {\small $2$};
\node (g) at (0,-1cm) [below] {\small $\Gamma_7$};
\end{tikzpicture}~~
\begin{tikzpicture}[scale=.8]
\draw [semithick,blue] (0,0) circle (1);
\draw [->,blue] (0,1) arc (90:180:1cm);
\draw [-,blue](-.717,.717cm)--(0,-1);
\draw [->,blue](-.717,.717cm)--(-0.717/2,0.717/2-1/2);
\draw [-,blue](.5,.866cm)--(0,-1);
\draw [->,blue](.5,.866cm)--(0.25,0.866/2-0.5);
\node (p) at (-.717,.717cm) [left] {\small $1$};
\node (q) at (-.2,-.8cm) [ ] {\small $2$};
\node (r) at (.5,.866cm) [right] {\small $3$};
\node (g) at (0,-1) [below] {\small $\Gamma_8$};
\end{tikzpicture}~~
\begin{tikzpicture}[scale=.8]
\draw [semithick,blue] (0,0) circle (1);
\draw [->,blue] (0,1) arc (90:180:1cm);
\draw [->,blue](0,1cm)--(0,.5cm);
\draw [-<,blue](0,-1cm) to [out=135,in=-90] (-2/7,-.25cm);
\draw [-,blue](-2/7,-.25cm) to [out=90,in=-135] (0,0.5cm);
\draw [-<,blue](0,-1cm) to [out=45,in=-90] (2/7,-.25cm);
\draw [-,blue](2/7,-.25cm) to [out=90,in=-45] (0,0.5cm);
\node (p) at (0,0.5cm) [left] {\small $1$};
\node (r) at (0,1cm) [above] {\small $1$};
\node (s) at (0,-1cm) [above ] {\small $2$};
\node (g) at (0,-1cm) [below] {\small $\Gamma_9$};
\end{tikzpicture}
\caption{Examples of graphs, the $|\textrm{Aut}\,\Gamma|$ factor for them are: 24, 4, 2 and 2, 4, 3, 2, 1, 1.} \label{ex_gph_fig}
\end{center}
\end{figure}
The graph differential acts through $1.\,$shrinking an edge running from internal vertices $i$ to $j$ (assuming $i<j$), naming the new vertex $i$ and decreasing the labels of the vertices after $j$ by one;
$2.\,$collapsing two adjacent peripheral vertices together and decreasing the labels of the ensuing peripheral vertices by one; $3.\,$shrinking an edge connecting one internal vertex and one peripheral vertex, and decreasing the labels of the ensuing internal vertices by one; see fig.\ref{dec_graph_diff_fig}. All three operations carry a sign factor given in fig.\ref{sign_fac_fig}.

Thus for the graphs of fig.\ref{ex_gph_fig},
\bea&&\hspace{2.4cm} \partial\Gamma_1=6\Gamma_3,~~~\partial\Gamma_2=2\Gamma_3;\nn\\
&&\partial\Gamma_4=-2\Gamma_8,~~~\partial\Gamma_5=4\Gamma_8,~~~\partial\Gamma_6=-3\Gamma_8+3\Gamma_9,~~~\partial\Gamma_7=2\Gamma_9.\label{gph_diff_ex}\eea

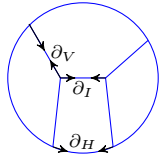
\begin{figure}[h]
\begin{center}
\begin{tikzpicture}[scale=1]
\draw [blue] (0,0) circle (1);

\draw [-,blue] (-.717,.717)--(-.3,0)
[-,blue] (-.3,.0)--(.3,0)
[-,blue] (.3,0)--(.866,.5)
[-,blue] (-.3,0)--(-.4,-.9165)
[-,blue] (.3,0)--(.4,-.9165);

\draw [->](-.3,.0)--(-.1,0);
\draw [->](.3,.0)--(.1,0);
\draw [->](-.3,0)--(-.439,.239); 
\draw [->](-.717,.717)--(-.508,.358);
\draw [->](-.4,-.9165)--(-.2,-.98); 
\draw [->](.4,-.9165)--(.2,-.98);

\node (a) at (0,-.85) {\scriptsize{$\partial_{H}$}};
\node (b) at (0,-.15) {\scriptsize{$\partial_{I}$}};
\node (b) at (-.25,.33) {\scriptsize{$\partial_{V}$}};
\end{tikzpicture}
\caption{Differential of a graph}\label{dec_graph_diff_fig}
\end{center}
\end{figure}

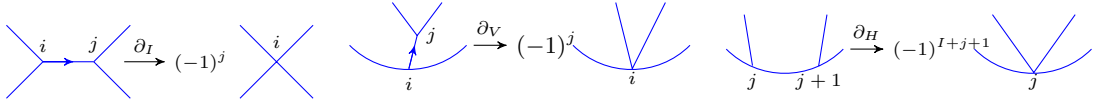
\begin{figure}[h]
\begin{center}
\begin{tikzpicture}[scale=1.35]
\draw [-,blue] (1,1) -- (1.5,1)
[->,blue] (1,1) -- (1.3,1);
\draw [-,blue] (.64,1.36)--(1,1) node [above] {{\color{black}\scriptsize $i$}};
\draw [-,blue] (1,1) -- (.64,0.64);
\draw [-,blue] (1.86,1.36)--(1.5,1) node [above] {{\color{black}\scriptsize $j$}};
\draw [-,blue] (1.5,1) -- (1.86,0.64);
\draw [->] (1.5+.3,1)--(1.5+.7,1) node[right] {\scriptsize{$(-1)^{j}$}};
\node  at (1.5+.5,1.15){\scriptsize${\partial_I}$};
\draw [-,blue] (.64+2.3,1.36)--(1.36+2.3,.64)
(.64+2.3,.64)--(1.36+2.3,1.36);
\node at (1+2.3,1.2){\scriptsize $i$};
\end{tikzpicture}~~~
\begin{tikzpicture}[scale=1.1]
\draw [-,blue] (-.717,-.717) arc (225:315:1)
 (.1,-.6)--(0,-1) node[below] {{\color{black}\scriptsize $i$}}
 (.1,-.6)--(.4,-.2)
 (-.2,-.2)--(.1,-.6) node[right] {{\color{black}\scriptsize $j$}};
\draw [->,blue] (0,-1)--(.07,-.72) ;
\draw [->] (-.717+1.5,-.717)--(-.717+1.9,-.717) node[right] {\small{$(-1)^{j}$}};
\node  at (-.717+1.7,-.517){\scriptsize${\partial_V}$};
\draw [-,blue] (-.717+2.7,-.717) arc (225:315:1)
 (-.2+2.7,-.2)--(0+2.7,-1)
 (.4+2.7,-.2)--(0+2.7,-1);
\node (l) at (0+2.7,-1.1165) {\scriptsize{$i$}};
\end{tikzpicture}~~~
\begin{tikzpicture}[scale=1.1]
\draw [-,blue] (-.717,-.717) arc (225:315:1)
 (-.5,-.3)--(-.4,-.9165)
 (.5,-.3)--(.4,-.9165);
\node (j) at (-.4,-1.1165) {\scriptsize{$j$}};
\node (k) at (.4,-1.1165) {\scriptsize{$j+1$}};
\draw [->] (-.717+1.5,-.717)--(-.717+1.9,-.717) node[right] {\scriptsize{$(-1)^{I+j+1}$}};
\node  at (-.717+1.7,-.517){\scriptsize${\partial_H}$};
\draw [-,blue] (-.717+3,-.717) arc (225:315:1)
 (-.5+3,-.3)--(0+3,-1)
 (.5+3,-.3)--(0+3,-1);
\node (l) at (0+3,-1.1165) {\scriptsize{$j$}};
\end{tikzpicture}
\caption{Sign factor for the differential, here $I$ is the number of
internal vertices. In the third picture, if $j$ is the
$l^{th}$-the last-vertex on the circle, then we rename the new vertex 1 with
sign factor $(-1)^{I+l+1}$. The overall minus sign from that of ref.\cite{WilsonLoop} is due to the difference in labelling schemes.}\label{sign_fac_fig}
\end{center}
\end{figure}

\subsection{Recipe for the CE-Graph Correspondence}\label{sec_RftCGC}
We shall only give the recipe, the proof may be found in refs.\cite{ConantVogt,Hamilton,WilsonLoop}, in particular, see ref.\cite{BV_lecture}, where the graph complex was presented as a polynomial and the proof was written in a few strokes.

 Let $\mathbbm{R}$ be a non-negatively graded vector space with constant symplectic form $\Omega$ of degree $\nb$, and $\BB{V}$ is a fixed graded vector space. In the following $f$ are formal polynomials on $\mathbbm{R}$ vanishing at the origin, while $g$ are $\textrm{Mat}(\BB{V})$-valued formal polynomials on $\mathbbm{R}$. But for now, one can consider the $g$'s entry by entry and treat them on the same footing as $f$.
Given an element in the extended CE chain $c_{p,q}$, written as
\bea (f_1,f_2,\cdots,f_p)\otimes [g_1|g_2|\cdots|g_q],\label{ECE_chain}\eea
write the formal product
\bea F=\big(t_1f_1(\xi_1)\big)\cdots\big(t_pf_p(\xi_p)\big)\big(s_1g_1(\eta_1)\big)\cdots\big(s_qg_q(\eta_q)\big),\nn\eea
where $t_i,~i=1,\cdots, p$ is of degree\footnote{The entire discussion about CE-graph correspondence is valid only for $\nb$-even, but we have retained $\nb$ here for the coherence of notation.} $\nb+1$ and $s_i,~i=1,\cdots, q$ is of degree 1 (which explains the suspension operation, see first paragraph of sec.\ref{sec_CECfGLA}).
Furthermore, both $\xi_i$ $\eta_i$ are coordinates of $\mathbbm{R}$, but given different names to distinguish the vertices.

Let $[\Gamma_{p,q}]$ be a representative of an equivalence class of graphs with $p$ internal and $q$ peripheral vertices. Let $E_{ij;}$ be an edge from internal vertices $i$ to $j$, $E_{i;j}$ be an edge from internal vertex $i$ to peripheral vertex $j$ and finally $E_{;ij}$ an edge from peripheral vertex $i$ to $j$. With each type of vertices we associate a differential
\bea E_{ij;}\to \beta_{ij;}=\big(\Omega^{-1}\big)^{AB}\frac{\partial}{\partial \xi_{i}^A}\frac{\partial}{\partial \xi_{j}^B},~~
E_{i;j}\to \beta_{i;j}=\big(\Omega^{-1}\big)^{AB}\frac{\partial}{\partial \xi_{i}^A}\frac{\partial}{\partial \eta_{j}^B},~~
E_{;ij}\to \beta_{;ij}=\big(\Omega^{-1}\big)^{AB}\frac{\partial}{\partial \eta_{i}^A}\frac{\partial}{\partial \eta_{j}^B}.\nn\eea
The fact that $\beta_{\cdot;\cdot}$ is anti-symmetric under exchange of $i,j$ (this requires $\nb$ be even) reflects the equivalence relation concerning the flipping of an edge in the previous section.

Form the product of operators
\bea \tilde\beta_{\Gamma}=\frac{1}{|\textrm{Aut}\,\Gamma|}\partial_{s_q}\cdots \partial_{s_1}\partial_{t_p}\cdots \partial_{t_1}~\prod_{E_{i;j}}\beta_{{i;j}}~\prod_{E_{;ij}}\beta_{{;ij}}~\prod_{E_{ij;}}\beta_{{ij;}}\nn\eea
where $|\textrm{Aut}\,\Gamma|$ is the order of automorphism group of the vertices\footnote{Usually, when one computes Feynman diagrams one also includes a factor of $1/p!$ for $p$ edges running between the same pair of vertices. In fact, this factor is implicitly included in the current formalism when one applies the operators $\beta_{\cdot;\cdot}$} of $\Gamma$.

The graph chain associated with the CE chain Eq.\ref{ECE_chain} is the formal sum over equivalence classes of graphs with $p$ internal and $q$ peripheral vertices
\bea (f_1,f_2,\cdots,f_p)\otimes [g_1|g_2|\cdots|g_q]\Rightarrow \sum_{[\Gamma_{p,q}]}~\big(\tilde\beta_{\Gamma_{p,q}}F\big)[\Gamma_{p,q}]\Big|_{t=s=\xi=\eta=0}+\textrm{perm},\label{recipe}\eea
where perm denotes the permutations among $f$'s and cyclic permutations among $g$'s with the sign factor as in Eq.\ref{sym_prop} and Eq.\ref{cyc_sym_prop}.

\emph{We shall denote collectively all the above operations as $\beta$}, forming one way of the correspondence Eq.\ref{CE_Gph_correspondence}. Trivially, by restricting $q=0$, all the previous assignments yield a map with values in the non-extended graph complex.
\begin{remark}We also observe that the definition of $\deg c^{\sbullet,\sbullet}$ in Eq.\ref{deg_cochain} is nothing but the (mod 2) degree of the operator $\beta$.\end{remark}

Here is the first main theorem of the paper concerning weight-systems for knot invariants.
Let $(\M,\Omega)$ be an even degree symplectic $NQ$-manifold, with an $\Omega$ preserving homological vector field $Q$, and $\E$ a graded vector bundle over $\M$ with a lift $Q^{\uparrow}$. Pick an exponential map which is a global section of $\FR{B}/SP(\mathbbm{R})$ (e.g. Eq.\ref{exponent_shifted}), construct the quantities ${\widehat Q}$ and ${\widehat T}$ as in Eqs.\ref{def_hat}, \ref{hat_T}. Since $Q$ preserves $\Omega$ and the local model $\mathbbm{R}$ inherits a symplectic structure from $\M$, the vector field $\widehat Q$ will have a Hamiltonian lift $\Theta$. We form the extended chain
\bea c=\sum_{p+q=m}c_{p,q},~~~c_{p,q}=\frac{1}{p!q}(\underbrace{\Theta_3,\cdots,\Theta_3}_p)\otimes \Tr[\underbrace{{\widehat T}_1|\cdots|{\widehat T}_1}_q],\nn\eea
where the subscripts 3 and 1 denote the cubic and linear term in the expansion of $\Theta$ and ${\widehat T}$ in terms of $\xi$. We have

\begin{theorem}\label{main_A} The graph chain $\beta c$ by applying the above recipe to $c$ is closed with values in the ring $H_Q(\M)$. And the graph cohomology class is independent of the choice of the connection in defining the exponential map, as well as the trivialization of $\E$. Thus we have a well-defined weight-system for chord diagrams valued in $H_Q(\M)$.
\end{theorem}
\begin{proof}
By reading Eq.\ref{MC} at order 2, we see that $Q\circ\Theta_3=0$ modulo an action of $\FR{sp}(\mathbbm{R})$ generated by $\Theta_2$, and by reading Eq.\ref{MC_adv} at order 1, we have $Q\circ {\widehat T}_1=0$ modulo the same action of $\FR{sp}(\mathbbm{R})$. Thus the coefficient of each graph in $\beta c$ is $Q$-closed, due to the $SP(\mathbbm{R})$ invariance of the construction. That the graph chain $\beta c$ is closed is because $\beta$ is a map of complexes: $\partial\beta c=\beta\partial c$, and $\partial c$ is $Q$-exact by using again Eqs.\ref{MC}, \ref{MC_adv}. Thus $\partial\beta c=\beta\,Q\circ \tilde c$ for certain $\tilde c$, and finally $\partial\beta c=Q\circ \beta\tilde c$ descending to zero in $H_Q(\M)$. The proof of the connection independence is similar to the proof of prop.\ref{prop_connection_ind}; the trivialization independence is proved by using Eq.\ref{gauge_transform_improve} and the property of trace\qed
\end{proof}

\subsection{The Dual Story}
As the graph chains are formal linear combinations of graphs, the dual graph cochains are written as a formal linear combination
\bea b=\sum_{[\Gamma]} ~b_{\Gamma}[\Gamma]^*,\label{gph_cochain}\eea
where the 'cographs' are defined through the obvious paring $\langle\left[\Gamma\right]^*,\left[\Gamma'\right]\rangle=\pm1$ if $[\Gamma]=\pm[\Gamma']$ and zero otherwise.

From a graph cochain $b$, we can construct a CE cochain, denoted $\beta^{\dagger} b$ using the recipe
\bea (\beta^{\dagger} b)((f_1,f_2,\cdots,f_p)\otimes [g_1|g_2|\cdots|g_q])=\sum_{[\Gamma]}b_{\Gamma}~\big(\beta_{\Gamma}F\big)\Big|_{t=s=\xi=\eta=0},\label{recipe_dual}\eea
where the lhs denotes the evaluation of $\beta^{\dagger} b$ on $(f_1,f_2,\cdots,f_p)\otimes [g_1|g_2|\cdots|g_q]$.

We are now in possession of all the technical tools to give the second main theorem of the paper. Taking $q=0$,
\begin{theorem}\label{main_B}
Let $\M$ be an even degree symplectic $N$-manifold, then from any graph cochain made of tri-valent graphs, one can construct a covariant cocycle of Lie algebra of Hamiltonian function on $\M$, with values in $C^{\infty}(\M)$. Besides, the cohomology class of the resulting CE cocycle is independent of the choice of connections in defining the exponential map.
\end{theorem}
\begin{proof}
The procedure is to use recipe \ref{recipe_dual} to convert the graph cocycle first to a CE cocycle of Hamiltonian functions on a flat space $\mathbbm{R}$ of the same dimension as $\M$. The tri-valent condition, guarantees amongst other things that the CE cochain is basic w.r.t $SP(\mathbbm{R})$. Then one uses the recipe \ref{prop_cochain_convert} to convert the type 1 cocycle to a cocycle of type $1^\prime$. The connection independence is demonstrated in prop.\ref{prop_connection_ind}\qed
\end{proof}

\section{Examples}\label{sec_EX}
\begin{example}\emph{Type 1 CE Cocycles}\\
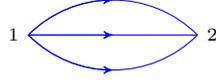
\begin{figure}[h]
\begin{center}
\begin{tikzpicture}[scale=0.9]
\draw [-,blue] (0:0cm) -- (0:2.5cm);
\draw [->,blue] (0:0cm) -- (0:1.25cm);

\draw [-,blue](0:2.5cm) arc (45:135:1.7675cm);
\draw [->,blue](0:0cm) arc (135:90:1.7675cm);

\draw [-,blue](0:0cm) arc (225:315:1.7675cm);
\draw [->,blue](0:0cm) arc (225:270:1.7675cm);

\node (p) at (0,0) {};
\node (q) at (2.5,0) {};

\draw (p) node [left] {\scriptsize{1}}
 (q) node [right] {\scriptsize{2}};
\end{tikzpicture}\caption{The simplest cocycle}\label{cocycle_fig}
\end{center}
\end{figure}
It can be shown that if in a graph cochain Eq.\ref{gph_cochain}, $b_{\Gamma}\neq0$ only for 3-valent graphs, then the graph cochain is a cocycle. By applying the recipe Eq.\ref{recipe_dual} to such a cocycle, one obtains a type 1 CE cocycle.

Take the graph cochain defined by fig.\ref{cocycle_fig}, that is, $b_{\Gamma}=0$ except for $b_{\ominus}=1$. The prescription gives the following cochain $c^2=\beta^{\dagger}[\ominus]^*$ (setting as in sec.\ref{sec_RftCGC})
\bea c^2(f_1,f_2)=\frac{1}{6}(-1)^{|f_1|}\big(f_1\overleftarrow{\partial}_{ABC}\big)(\Omega^{-1})^{AD}(\Omega^{-1})^{BE}(\Omega^{-1})^{CF}\big(\partial_{FED}f_2\big),\label{cocycle_2pt}\eea
where the entire rhs is evaluated at the origin. The sign factor $(-1)^{|f_1|}$ follows from the recipe and is crucial for $c^2$ to satisfy Eq.\ref{sym_prop}. A skeptical reader may wish to check that $c^2$ is closed for himself.
\end{example}
\begin{example}\emph{Type 1${}^\prime$ CE Cocycles}\\
\noindent For clarity, we focus on the smooth case $\M=M^{2n}$, with symplectic form $\Omega$. We use the exponential map Eq.\ref{exponent_shifted} for a symplectic manifold, thus $\FR{G}$ is valued in $\FR{sp}$. We define $u_f$ to be the vector field generated by $f\in C^{\infty}(M)$. Since $\BB{R}^{2n}$ has a symplectic form $\Omega(\textsl{x})$, which is pulled back from $M^{2n}$ by the exponential map. Then let $\hat f$ the Hamiltonian lift of $\widehat{u_f}$ w.r.t $\Omega(\textsl{x})$.
One has the freedom to set the constant term of $\hat f$ to be zero; this done, $\hat f$ vanishes at the origin as $\xi^2$.

Using the prescription of Eq.\ref{cochain_convert}, we construct a covariant cochain of type $1^{\prime}$
\bea \FR{c}^2(f,h)=c^2(\hat{f},\hat{h})=\frac16\big(f^{\mu\nu\rho}
+f^{\lambda}\SF{R}_{\lambda}^{~\mu\nu\rho}\big)\big(h_{\mu\nu\rho}+h^{\kappa}\SF{R}_{\kappa\mu\nu\rho}\big)\in C^{\infty}(M),\label{cocycle_2pt_cov}\eea
where $f_{i_1\cdots i_n}=\nabla_{i_1}\cdots\nabla_{i_{n-1}}\partial_{i_n}f\big|_{\textsl{x}}$ and all indices are raised (resp. lowered) with $\Omega^{-1}(\textsl{x})$ (resp. $\Omega(\textsl{x})$).

To check explicitly Eq.\ref{cocycle_2pt_cov} is closed requires some effort, the key step is to show
\bea&&\frac16\xi^{\ga}\xi^{\gb}\xi^{\gc}\big(f^{\rho}\nabla_{\rho}h_{\ga\gb\gc}\big)=\frac16\xi^{\ga}\xi^{\gb}\xi^{\gc}\Big\{-f_\ga^{~\rho}h_{\rho\gb\gc}-2f^{\rho}\SF{R}_{\rho\ga~\gb}^{~~\,\sigma}h_{\sigma\gc}
+\nabla_\ga\nabla_\gb(f^{\rho}h_{\rho\gc})-f_{\ga\gb}^{~~\,\rho}h_{\rho\gc}\nn\\
&&\hspace{5.3cm}-f_\gb^{~\rho}h_{\ga\rho\gc}
-f_\ga^{~\rho}\SF{R}_{\rho\gb~\gc}^{~~\;\sigma}h_\sigma-f^{\rho}(\nabla_{\ga}\SF{R}_{\rho\gb~\gc}^{~~\;\sigma})h_\sigma-f^{\rho}\SF{R}_{\rho\gb~\gc}^{~~\;\sigma}h_{\ga\sigma}\Big\}\nn\eea
by commuting $f^{\rho}\nabla_{\rho}$ over the derivatives $\nabla_{\ga}\nabla_{\gb}\partial_{\gc}$; this will then lead to
\bea f^{\mu}\partial_{\mu}\hat{h}-h^{\mu}\partial_{\mu}\hat{f}+\{\hat{f},\hat{h}\}-\widehat{\{f,h\}}=0,\nn\eea
where the first curly brace $\{-,-\}$ is the Poisson bracket on $C^{\infty}(\BB{R}^{2n})$, the second on $C^{\infty}(M^{2n})$. This equation is the Hamiltonian version of Eq.\ref{key_id}.
\end{example}

\begin{example}\emph{Type 2 cycles of extended graph complex from extended CE cycles}\\
\noindent Let $\FR{g}$ be a Lie algebra $\FR{su},\,\FR{sp}$ or $\FR{so}$. Let $\M=\FR{g}[1]$ with coordinates $\ell^{\ga}$. Let the matrices $T_{\ga}$ be a representation of $\FR{g}$, and the killing metric $\eta_{\ga\gb}=\Tr[T_{\ga}T_{\gb}]$ plays the role of the symplectic form on $\FR{g}[1]$. Finally let $\Theta=1/6\eta_{\gc\kappa}f^{\kappa}_{~\ga\gb}\ell^{\gc}\ell^{\ga}\ell^{\gb}$, then one can check that
\bea c=\frac14\Tr[T_{\ga}T_{\gb}T_{\gc}T_{\gd}]~(~)\otimes [\ell^{\ga}|\ell^{\gb}|\ell^{\gc}|\ell^{\gd}]
+\frac{1}{3}\Tr[T_{\ga}T_{\gb}T_{\gc}]~(\Theta)\otimes [\ell^{\ga}|\ell^{\gb}|\ell^{\gc}]
+\frac{1}{2}\frac{1}{2!}\Tr[T_{\ga}T_{\gb}]~(\Theta,\Theta)\otimes [\ell^{\ga}|\ell^{\gb}]\nn\eea
is a cycle in the extended CE complex.

Apply the recipe Eq.\ref{recipe} to $c$
\bea\beta c=\frac{d_GC_2(G)}{2}\big(\frac{1}{4}[\Gamma_5]+\frac{1}{3}[\Gamma_6]-\frac{1}{2}[\Gamma_7]\big)
-\frac{d_rC_2^2(r)}{4}\big([\Gamma_5]+2[\Gamma_4]\big)~,\nn\eea
where $d_r (d_G)$ is the dimension of the (adjoint) representation. One may use Eq.\ref{gph_diff_ex} to check that each combination in the two braces is a graph cycle. This is the \emph{Lie algebra weight-system}.
\end{example}
\subsection{Weights for Graphs Valued in $Q$-cohomology}\label{sec_WfGViQc}
This is the central (and non-trivial) application of our result and therefore deserves a separate section.

Let $M^{4n}$ be a holomorphic symplectic manifold with complex coordinates $x^i,x^{\bar i}$. The holomorphic symplectic form is $\Omega_{ij}dx^idx^j$, and one can always pick a torsionless connection simultaneously preserving the complex structure and the symplectic structure, thus in the complex basis, only the combination $\Gamma^i_{jk}$, $\Gamma^{\bar i}_{\bar j\bar k}$ is non-zero.

Consider the $NQ$-manifold
\bea \M=T_{(0,1)}[1]M,\label{first_choice}\eea
where $T_{(0,1)}$ denotes the anti-holomorphic tangent bundle. Locally, the coordinates are
\bea \deg 0:~x^i,~x^{\bar i};~~~~~~\deg 1:~v^{\bar i},\nn\eea
and the 2-form $\Omega_{ij}$ is assigned degree 2. The $Q$-vector field
\bea Q=v^{\bar i}\frac{\partial}{\partial x^{\bar i}}\label{Dolbeault}\eea
corresponds to the Dolbeault differential.

The local model $\mathbbm{R}$ for $\M$ is $\mathbbm{R}=\BB{C}^{4n}\times \BB{C}^{2n}[1]$, and we denote the flat coordinates as $\xi^i,\,\xi^{\bar i}$ and $\nu^{\bar i}$. One may proceed to apply the exponential map, but since eventually we shall only need a CE cochain of Hamiltonian function in the variable $\xi^i$, with symplectic form $\Omega_{ij}$ (see the next remark), we may set $\xi^{\bar i}=0$ and $\nu^{\bar i}=0$ (not $v^{\bar i}$).

The exponential map is the holomorphic half of Eq.\ref{exponent_shifted},
\bea &&\phi_{hol}^i=x^i+\xi^i-\frac12\Gamma^i_{mn}\xi^m\xi^n+\left (-\frac16\partial_j\Gamma^i_{mn}+\frac13\Gamma^i_{pj}\Gamma^p_{mn}
-\frac{1}{24}\textsf{R}^i_{~jmn} \right ) \xi^j\xi^m\xi^n+{\cal O}(\xi^4),\nn\\
&&\hspace{3.5cm}\textsf{R}^{i}_{~jmn}=(\Omega^{-1})^{li}R_{lj~n}^{~~k}\Omega_{km},\nn\eea
and $\phi_{hol}^{\bar i}=x^{\bar i}$, $\phi_{hol}^{v^{\bar i}}=v^{\bar i}$. For more in depth discussion of the holomorphic exponential maps, see ref.\cite{Kapranov}.

One may compute the quantity (where the Poisson bracket is taken with $\Omega^{-1}({x})$),
\bea {\widehat Q}=\{v^{\bar i}\Theta_{\bar i},-\},~~~~~~\Theta_{\bar i}=\frac{1}{6}\xi^i\xi^j\xi^k\textsf{R}_{\bar iijk}+{\cal O}(\xi^4).\nn\eea
And $\Theta$ satisfies the Maurer-Cartan equation
\bea \bar\partial_{[\bar i}\Theta_{\bar j]}=-\{\Theta_{\bar i},\Theta_{\bar j}\},~~~\textrm{i.p.}~~~\bar\partial_{[\bar i}\Theta_{\bar j]}\Big|_{\xi^3}=0.\label{Maurer_Cartan}~.\eea
In fact, if $M$ is hyperK\"ahler, $\Theta$ can be computed to all orders in $\xi$ (see ref.\cite{oddCS})
\bea
\Theta_{\bar i}(\xi)=\sum_{n=3}^{\infty}~\frac{1}{n!}\nabla_{l_4}\cdots\nabla_{l_n}\textsf{R}_{\bar il_1l_2l_3}\,\xi^{l_1}...\xi^{l_n}~.\label{def_Theta}\eea
\begin{remark}
 It may seem that we are not following our own recipe in theorem \ref{main_A} as we only have a holomorphic symplectic form $\Omega_{ij}$. To go by the rules, one has to take $\M=T^{(0,1)}[2]T_{(0,1)}[1]M$ to give $x^{\bar i}$ and $v^{\bar i}$ their conjugate momenta. But it is fairly clear that, this will lead to a Hamiltonian lift of $\widehat Q$ that includes $\Theta$ above, plus terms at most linear in the momenta conjugate to $\xi^{\bar i},\,\nu^{\bar i}$. These terms will vanish once we plug the Hamiltonian function into a tri-valent graph. Thus setting $\xi^{\bar i},\,\nu^{\bar i}$ to zero from the beginning is a slight short-cut we take.
\end{remark}

So far at order $\xi^3$ the effect of the Grothendieck connection has not shown up, because we notice that $\Theta_{\bar i}$ starts at order $\xi^3$, namely the pushforward of $\bar\partial$ has no $\xi^0\partial_{\xi}$ term and hence the effect of the Grothendieck connection comes in only at the $4^{th}$ order, which is not picked up by the tri-velent graphs. To see the indispensability of the Grothendieck connection, one needs to consider a situation where the $0^{th}$ term in the expansion of the pushforward of $Q$ is nonzero.

To this end let us twist the $Q$-vector field Eq.\ref{Dolbeault} slightly. Suppose there is a group $G$ acting on $M$ through a holomorphic moment map $\mu_{\ga}$, $\ga=1\sim \textrm{rk}_G$, $\partial_{\bar i}\mu_{\ga}=0$. Let now $\M$ be
\bea T_{(0,1)}[1]M\times \FR{g}[1],\nn\eea
and we let the coordinates of $\FR{g}[1]$ be $\ell^{\ga}$ of degree 1. The twisted $Q$-vector field is now
\bea Q=v^{\bar i}\frac{\partial}{\partial x^{\bar i}}+\ell^{\ga}(\partial_j\mu_{\ga})(\Omega^{-1})^{ji}\frac{\partial}{\partial x^{i}}
-\frac{1}{2}f^{\gc}_{~\ga\gb}\ell^{\ga}\ell^{\gb}\frac{\partial}{\partial\ell^{\gc}},\label{Q_RW}\eea
where the middle term is going to lead to (amongst other things) a non-zero $0^{th}$ term $\big(\ell^{\ga}(\partial_j\mu_{\ga})(\Omega^{-1})^{ji}\big|_{\SL{x}}\big)\partial_{\xi^{i}}$.

One can again compute the quantity $\widehat{Q}$
\bea&&~~\widehat Q=\{v^{\bar i}\Theta_{\bar i},-\}+\{\ell^{\ga}\textsf{M}_{\ga},-\},\nn\\
\textrm{where}&&\textsf{M}_{\ga}=\frac12(\nabla_i\partial_j\mu_{\ga})\xi^i\xi^j
+\frac16(\nabla_i\nabla_j\partial_k\mu_{\ga}-(\partial_l\mu_{\ga})\textsf{R}^l_{\;ijk})\xi^i\xi^j\xi^k+{\cal O}(\xi^4),\label{expansion_M}\eea
as usual all the terms involving $\mu_{\ga},\,\SF{R}$ etc are evaluated at $\textsl{x}$.
In this expression, the effect of the Grothendieck connection is to correct the coefficient of the curvature term from $-1/24$ to $-1/6$, which is crucial for the
\begin{proposition}
Under a change of the connection $\Gamma^i_{ij}\to\Gamma^i_{ij}+\gamma^i_{ij}$ (but retaining the torsionless and $J,\,\Omega$-preserving property), we have up to order $\xi^3$
\bea \delta_{\Gamma}(v^{\bar i}\Theta_{\bar i}+\ell^{\ga}\mathsf{M}_{\ga})=Q\circ\Psi+\{\ell^{\ga}\mathsf{M}_{{\ga}2},\Psi\}~,
~~~\Psi=\frac16\gamma^l_{ik}\Omega_{lj}\xi^i\xi^j\xi^k~,\nn\eea
where $\textsf{M}_{\ga2}$ is the quadratic term in the expansion of $\textsf{M}_{\ga}$ in Eq.\ref{expansion_M}.
\end{proposition}
This result is a special case of Eq.\ref{change_hat_u}, truncated at order $\xi^3$. One can of course do a direct computation, in fact, this was how we first came to realize the necessity of the Grothendieck connection in ref.\cite{RWCS}, after accidentally discovering that shifting $-1/24$ to $-1/6$ made everything work.

If $c^k$ is a cocycle constructed from tri-valent graphs, then by evaluating the corresponding $\FR{c}^k$ on the functions $\Theta+\textsf{M}$, one obtains an element
\bea \FR{c}^k(\Theta+\textsf{M},\cdots,\Theta+\textsf{M})\in \bigoplus_{p+q=k}\Omega_M^{(0,p)}\otimes \wedge^q \FR{g}^*,\nn\eea
which is annihilated by the differential Eq.\ref{Q_RW}. These are the equivariant Rozansky-Witten classes. It is equivariant in the sense that, in the real setting, the complex $\Omega_M^p\otimes\wedge^q \FR{g}^*$ is a model for the de Rham complex of $M\times_{G}EG$ with $EG$ being the universal $G$-bundle. By applying the result of ref.\cite{Bott}, one can turn this complex into the more familiar Cartan model of equivariant cohomology.

To see what is special about the holomorphic setting chosen above, let us take $M^{2n}$ to be a symplectic manifold with symplectic form $\Omega_{\mu\nu}$. Pick as before a torsionless connection $\Gamma^{\mu}_{\nu\rho}$ preserving $\Omega$, and
$Q=v^{\mu}{\partial _{\mu}}$
now corresponds to the de Rham differential. Using the exponential map Eqs.\ref{exponent_shifted}, \ref{exponent_graded} to compute $\phi^{-1}_* Q$ gives
\bea \phi^{-1}_* Q=v^{\mu}\frac{\partial}{\partial x^{\mu}}+\frac18v^{\rho}\big(\textsf{R}^{\mu}_{~\rho\ga\gb}-2\textsf{R}_{\ga\rho~\gb}^{~~\;\mu}\big)\xi^{\ga}\xi^{\gb}\frac{\partial}{\partial \xi^{\mu}},
~~~~\textsf{R}_{\mu\ga\gb\gc}=R_{\mu\ga~\gc}^{~~\;\kappa}\Omega_{\kappa\gb}.\nn\eea
The second term fails to be Hamiltonian. This is simply because the vector field $v^{\mu}\partial_{\mu}$ does not preserve $\Omega_{\ga\gb}$ as $v^{\bar i}\partial_{\bar i}$ does $\Omega_{ij}$.

It was pointed out in ref.\cite{Sawon} that from any
holomorphic vector bundle $E$ over $M$ with connection $A$ and curvature $K$ (of type (1,1), naturally), one can
construct a representation for $Q$ Eq.\ref{Dolbeault} (we will not consider the twisted case). First denote the deg 0
coordinate of the fibre of $E$ as $z_{\ga}$ and its flat counterpart as $\zeta_{\ga}$.
The lift of $Q$ to $E$ is
\bea Q^{\uparrow}=v^{\bar i}\frac{\partial}{\partial x^{\bar i}}+v^{\bar i}(A_{\bar i})^{\ga}_{~\gb}z_{\ga}\frac{\partial}{\partial z_{\gb}},\nn\eea
corresponding to the (0,1)-covariant derivative $dx^{\bar i}\nabla_{\bar i}$.

In applying the exponential map Eq.\ref{exponent_graded}, we will expand around $z_{\ga}=0$ (not $\zeta_{\ga}$). We can compute
\bea \widehat T=\widehat{Q^{\uparrow}}-\widehat Q=-\big(v^{\bar i}\textsf{K}_{\bar i\gb}^{\ga}\big)\zeta_{\ga}\frac{\partial}{\partial\zeta_{\gb}},
~~~\textsf{K}_{\bar i\gb}^{\ga}=\xi^j (K_{\bar ij})^{\ga}_{~\gb}+\frac12\xi^j\xi^k(\BS{\nabla}_jK_{\bar ik})^{\ga}_{~\gb}+{\cal O}(\xi^3).\nn\eea
Again, if the manifold is hyperK\"ahler, $\textsf{K}$ can be computed to all orders
\bea\textsf{K}_{\bar i\gb}^{\ga}=\sum_{p=0}^{\infty}\frac{1}{(p+1)!}(\BS{\nabla}_{l_1}\cdots \BS{\nabla}_{l_p}K_{\bar il_{p+1}})^{\ga}_{~\gb}\xi^{l_1}\cdots\xi^{l_{p+1}}~,\nn \eea
$\textsf{K}$ satisfies a neat relation
\bea {\nabla}_{[\bar i}\textsf{K}_{\bar j]}=\textsf{K}_{[\bar i}\textsf{K}_{\bar
j]}-\{\Theta_{[\bar i},\textsf{K}_{\bar j]}\},~~~\textrm{i.p.}~~~{\nabla}_{[\bar i}\textsf{K}_{\bar j]}\Big|_{\xi^1}=0.\label{neat_relation}\eea

To complete the discussion, we use $\Theta$ and $\textsf{K}$ to form the cycle
\bea c=\sum_{p+q=m}c_{p,q},~~c_{p,q}=\frac{(-1)^q}{p!q}(\underbrace{v^{\bar i}\Theta_{\bar i},\cdots,v^{\bar i}\Theta_{\bar i}}_p)\otimes\Tr[\underbrace{v^{\bar i}\textsf{K}_{\bar i}|\cdots|v^{\bar i}\textsf{K}_{\bar i}}_q]\in\Omega^{(0,m)}_M,\nn\eea
and by applying the prescription of sec.\ref{sec_RftCGC}, we will obtain a graph chain with coefficients in $\Omega^{(0,m)}(M)$, and the boundary of this graph chain is $\bar\partial$-exact. This is the \emph{Rozansky-Witten weight-system}.

If one plugs this extended graph chain into a cochain made of chord diagrams, one obtains a cohomology class in $H^{m}_{\bar \partial}(M)$. See refs.\cite{Sawon} and \cite{WilsonLoop} for some more detailed calculation. The cohomology class is independent of the choice of the connections $\Gamma$ or $A$, these are the so called Rozansky-Witten classes.

Finally, to obtain even more examples, one may try to use the Courant algebroids and construct certain representations on them, though non-trivial examples are at the time being unknown. One may also take a symplectic manifold with some commuting circle action, and construct weight-systems similar to the twisted RW case (though dropping the Dolbeault part), and these are expected to carry information of the foliation of the manifold by these circle actions.

\appendix
\section{The Chevalley-Eilenberg Differential}
We give the complete expression of the differential of an extended CE cochain, valid for general $\nb$. The setting is as follows: $\mathbbm{R}$ is a non-negatively graded vector space, $f_i$ are the formal vector fields thereon. Fix a graded vector space $\BB{V}$, $g_i$ are formal polynomials on $\mathbbm{R}$ valued in $\textrm{Mat}(\BB{V})$. Just as in sec.\ref{sec_RftCGC}, we consider the $g$'s entry-wise, and treat them as mere polynomials on $\mathbbm{R}$.

The full differential is presented as a sum $\delta=\delta_I+\delta_V+\delta_H$, each of which is dual to the operators in fig.\ref{dec_graph_diff_fig}. The sign factors below are obtained from combining Eqs.\ref{CE_boundary}, \ref{CE_diff}, \ref{bar_diff} and \ref{g_action_B}. Let $c=\sum_{p,q}\,c^{p,q}$ be a cochain,
\bea
(\delta_I
c)(f_1,\cdots f_{p};~g_1,\cdots g_q)=\sum_{1\leq i<j\leq
p}(-1)^{s_{ij}}c((-1)^{|f_i|}[f_i,f_j],f_1\cdots \hat{i}\cdots
\hat{j}\cdots f_{p}; g_1,\cdots,g_q),\nn\\
(\delta_V
c)(f_1,\cdots f_{p};~g_1,\cdots g_q)
=-\sum_{1\leq i\leq p;1\leq j\leq
q}(-1)^{t_{ij}}(f_1,\cdots\hat{i}\cdots f_{p};
g_1\cdots g_{j-1},f_i\circ g_j,g_{j+1}\cdots g_q),\nn\\
(\delta_H
c^{p,q})(f_1,\cdots f_p;~g_1,\cdots g_{q})=-\sum_{1\leq j\leq q}(-1)^{u_{j}}(f_1,\cdots f_p;
g_jg_{j+1},g_{j+2}\cdots g_{q},g_1,\cdots g_{j-1})~,\label{CE_diff_full}\eea
\bea
s_{ij}&=&|\bar f_i|\sum_{k=1}^{i-1}|\bar f_k|+|\bar f_j|\sum_{k=1}^{j-1}|\bar f_k|+|\bar f_i||\bar f_j|~,\nn\\
t_{ij}&=&|\bar f_i|(\sum_{k=i+1}^{p}|\bar f_k|+\sum_{k=1}^{j-1}|\bar g_k|)+\sum_{k=1}^{p}|\bar f_k|+\sum_{k=1}^{j-1}|\bar g_k|~,\nn\\
u_{j}&=&\sum_{k=1}^p|\bar f_k|+|\bar g_j|+\sum_{k=1}^{j-1}|\bar g_k|\sum_{m=j}^q|\bar g_m|~,\label{sgn_CE_diff_full}\eea
where $|\bar f|=|f|+\nb+1$ and $|\bar g|=|g|+1$. The formidable signs disappear for most of our applications, since $f$, $g$ will mostly be of odd degree and $\nb$ even in the text.

\vskip1.5cm
\noindent{\emph{Note added}\\
\smallskip
\noindent
About two months after we posted our paper, an interesting paper \cite{Costello} by Costello appeared, which is also partly concerned with constructing $L_{\infty}$-structures from a curved $NQ$-manifold, and Costello termed these structures the 'curved $L_{\infty}$-structure'. We give in this appendix certain translation of Costello's formulation into our language to facilitate the comparison of the two papers.

In this paper, we constructed the Grothendieck connection following the idea of Bott and Haefliger because of its geometrical nature. However, a more common formulation of the Grothendieck connection $\FR{G}$, which is also the one adopted in ref.\cite{Costello}, is through the jet bundle $\mathscr{J}\to\M$. The Grothendieck connection $\FR{G}$ is regarded as a connection of this bundle. As $\FR{G}$ is flat, one has a differential $D$ by twisting the de Rham (or Dolbeault) differential on $\M$ by $\FR{G}$. Then the statement such as Eq.\ref{temp4} merely says that $\phi^{-1}_{\SL{x}}u$ is a parallel section on $\mathscr{J}\otimes T\M$.

The exponential map gives an identification $\mathscr{J}\sim \widehat{\textrm{sym}}^{\sbullet}(T_{\M}^*)$, where $\widehat{\textrm{sym}}^{\sbullet}(T_{\M}^*)$ is the completed symmetric algebra of $T^*_{\M}$. The previous differential $D$ can be passed onto $\widehat{\textrm{sym}}^{\sbullet}(T_{\M}^*)$ and making it a differential graded algebra over $\Omega^{\#}(M)$, where ${}^{\#}$ is the forgetful functor that forgets the differential.
The dga $\widehat{\textrm{sym}}^{\sbullet}(T_{\M}^*)$ has an ideal $\widehat{\textrm{sym}}^{>0}(T_{\M}^*)$, and this ideal is not preserved by $D$. In our paper, this simply translates to the property that $D$ may have a non vanishing $0^{th}$ Taylor coefficient. In this situation, we do not have an $L_{\infty}$-structure on $T[1]M$.

To proceed, one notices that $\Omega^{\#}(M)$ has a nilpotent ideal $I$ given by the elements of degree $>0$. Costello requires that those terms in $D$ that do not preserve $\widehat{\textrm{sym}}^{>0}(T_{\M}^*)$ be in $I$, and thus one gets an $L_{\infty}$-structure modulo $I$ (see def.2.1.1 in ref.\cite{Costello}). In our paper, we showed that in such situation one has an $L_{\infty}$-structure up to $Q$-exact terms, which is a slight refinement of the previous formulation by Costello. For example, for the Rozansky-Witten case, we have an $L_{\infty}$-structure up to $\bar\partial$-exact terms, while the equivariant RW case gives an $L_{\infty}$-structure up to a $Q$ given in Eq.\ref{Q_RW}. These structures are termed the 'curved $L_{\infty}$-structure', and Costello proves likewise the independence of the curved $L_{\infty}$-structure on the connection used in the exponential map, albeit in a fairly abstract way (lem.3.0.3).

Also in a more recent paper \cite{2015CMaPh.336..217V}, the RW weight system is also treated from the point of view of Atiyah classes. Our result gives more concrete formulae to the same construction. 


\providecommand{\href}[2]{#2}\begingroup\raggedright\endgroup

\end{document}